\documentclass[10pt,reqno]{amsart}
\usepackage{amsmath, amsthm, amscd, amsfonts, amssymb, graphicx, color, mathrsfs}
\usepackage{lineno}
\usepackage[bookmarksnumbered, colorlinks, plainpages]{hyperref}
\hypersetup{colorlinks=true,linkcolor=red, anchorcolor=green, citecolor=cyan, urlcolor=red, filecolor=magenta, pdftoolbar=true}
\modulolinenumbers[5]

\newtheorem{thm}{Theorem}[section]
\newtheorem{lemma}[thm]{Lemma}
\newtheorem{pro}[thm]{Proposition}
\newtheorem{cor}[thm]{Corollary}
\theoremstyle{definition}
\newtheorem{defn}[thm]{Definition}

\newtheorem{hyp}[thm]{Hypothesis}

\theoremstyle{remark}

\newtheorem{eg}[thm]{Example}

\numberwithin{equation}{section}

\newcommand{\ol}[1]{\overline{#1}}

\renewcommand{\hat}[1]{\widehat{#1}}
\renewcommand{\tilde}[1]{\widetilde{#1}}

\newcommand{\set}[1]{{\left\{#1\right\}}}
\newcommand{\pa}[1]{{\left(#1\right)}}
\newcommand{\sq}[1]{{\left[#1\right]}}
\newcommand{\gen}[1]{{\left\langle #1\right\rangle}}
\newcommand{\abs}[1]{{\left|#1\right|}}
\newcommand{\norm}[1]{{\left\|#1\right\|}}

\newcommand{\ssm}{\smallsetminus}

\newcommand{\ra}{\rightarrow}

\newcommand{\longra}{\longrightarrow}

\newcommand{\xra}{\xrightarrow}

\newcommand{\N}{\mathbb{N}}

\newcommand{\R}{\mathbb{R}}

\newcommand{\eqsys}[1]{{\left\{\begin{array}{ll}#1\end{array}\right.}}

\newcommand{\elle}{\operatorname{L}}

\newcommand{\tc}{\, \middle |\,}                                                

\newcommand{\con}{\operatorname{\mathscr{C}}}

\newcommand{\eps}{\varepsilon}

\newcommand{\T}{{\operatorname{\mathcal{T}}}}

\newcommand{\fcon}{\operatorname{\mathscr{FC}}}

\DeclareMathOperator{\borel}{{\operatorname{Borel}}}

\DeclareMathOperator{\lip}{\operatorname{Lip}}
\DeclareMathOperator{\diver}{\operatorname{div}}
\DeclareMathOperator{\sign}{\operatorname{sign}}
\DeclareMathOperator{\trace}{\operatorname{Tr}}

\begin{document}

\frenchspacing

\title[Sobolev spaces with respect to a weighted Gaussian measures]{Sobolev spaces with respect to a weighted Gaussian measures in infinite dimensions}

\author[S. Ferrari]{{S. Ferrari}}

\address[S. Ferrari]{Dipartimento di Matematica e Fisica ``Ennio de Giorgi'', Universit\`a del Salento, Via per Arnesano snc, 73100 Lecce, Italy.}
\email{\textcolor[rgb]{0.00,0.00,0.84}{simone.ferrari@unisalento.it}}




\date{\today}

\begin{abstract}
Let $X$ be a separable Banach space endowed with a non-degenerate centered Gaussian measure $\mu$ and let $w$ be a positive function on $X$ such that $w\in W^{1,s}(X,\mu)$ and $\log w\in W^{1,t}(X,\mu)$ for some $s>1$ and $t>s'$. In the present paper we introduce and study Sobolev spaces with respect to the weighted Gaussian measure $\nu:=w\mu$. We obtain results regarding the divergence operator (i.e. the adjoint in $\elle^2$ of the gradient operator along the Cameron--Martin space) and the trace of Sobolev functions on hypersurfaces $\set{x\in X\tc G(x) = 0}$, where $G$ is a suitable version of a Sobolev function.
\end{abstract} 

\maketitle

\section{Introduction}

Let $X$ be a separable Banach space with norm $\norm{\cdot}_X$, endowed with a non-degenerate centered Gaussian measure $\mu$. The associated Cameron--Martin space is denoted by $H$, its inner product by $\gen{\cdot,\cdot}_H$ and its norm by $\abs{\cdot}_H$. The spaces $W^{1,p}(X,\mu)$ and $W^{2,p}(X,\mu)$ for $p\geq 1$ are the classical Sobolev spaces of the Malliavin calculus (see \cite{Bog98}). The following hypothesis will be assumed throughout the paper:
\begin{hyp}\label{ipotesi1}
We will denote the weight function by $w$ and assume that $w(x)>0$, $\mu$-a.e. and 
\begin{enumerate}
\item $w\in W^{1,s}(X,\mu)$ for some $s> 1$ and it is a $(1,s)$-precise version (see section \ref{Notations and preliminaries});

\item $\log w\in W^{1,t}(X,\mu)$ for some $t>s'$, where $s'$ denotes the conjugate exponent of $s$ (this technical hypothesis will be explained later).
\end{enumerate}
We set
\[\nu:=w\mu.\]
\end{hyp}
We recall that $\nu$ is a Radon measure absolutely continuous with respect to $\mu$. The aim of this paper is to study the Sobolev spaces with respect to the measure $\nu$. In particular we study the divergence operator, i.e. the adjoint in $\elle^2$ of the gradient operator, and the traces on hypersurfaces $\set{x\in X\tc G(x) = 0}$ where $G:X\ra \R$ is a suitable version of a Sobolev function. The main results about Sobolev spaces with respect to the Gaussian measure can be found in \cite{Bog98}, while results about surface measures and traces of Sobolev functions in Gaussian Sobolev spaces can be found in \cite{FP91,Fey01} and in \cite{CL14}, respectively. 
We will assume the following hypotheses about the ``regularity'' of the hypersurfaces we work with:

\begin{hyp}\label{ipotesi2}
Let $G\in W^{2,q}(X,\mu)$ be a $(2,q)$-precise version (see Section \ref{Notations and preliminaries}) for every $q>1$ and assume
\begin{enumerate}
\item $\mu(G^{-1}(-\infty,0))>0$;

\item \label{Bastaaaaa} there exists $\delta>0$ such that $\abs{\nabla_H G}_H^{-1}\in\elle^q(G^{-1}(-\delta,\delta),\mu)$ for every $q>1$.
\end{enumerate}
\end{hyp}
Hypothesis \ref{ipotesi2}\eqref{Bastaaaaa} is classical, indeed see \cite{AM88}, \cite{Bog98} and \cite{CL14}, but it is difficult to prove. In this paper we are able to check Hypothesis \ref{ipotesi2}\eqref{Bastaaaaa} only for some notable functions $G$ (see Section \ref{Examples}).

One of our main results is an infinite dimensional weighted version of the divergence theorem, namely:

\begin{thm}\label{divergence theorem with traces}
Let $p\geq\frac{t}{t-s'}$ and $\{e_k\,|\, k\in\N\}$ be an orthonormal basis of $H$. For every $\varphi\in W^{1,p}(G^{-1}(-\infty,0),\nu)$ and $k\in\N$ we have
\[\int_{G^{-1}(-\infty,0)}\pa{\partial_k\varphi+\varphi\partial_k\log w-\varphi\hat{e}_k}d\nu=\int_{G^{-1}(0)}\trace_{G^{-1}(0)}(\varphi \partial_k G)\frac{w}{\abs{\nabla_H G}_H}d\rho.\]
Furthermore if $\Phi\in W^{1,p}(G^{-1}(-\infty,0),\nu; H)$ then
\[\int_{G^{-1}(-\infty,0)}\diver_\nu\Phi d\nu=\int_{G^{-1}(0)}\gen{\trace_{G^{-1}(0)}\Phi,\trace_{G^{-1}(0)}\nabla_H G}_H\frac{w}{\abs{\nabla_H G}_H}d\rho,\]
where $\trace_{G^{-1}(0)}\Psi=\sum_{n=1}^{+\infty}(\trace_{G^{-1}(0)}\psi_n)e_n$ if $\Psi\in W^{1,p}(G^{-1}(-\infty,0),\nu;H)$ and $\psi_n=\gen{\Psi,e_n}$.
\end{thm}
The functions $\{\hat{e}_k\,|\, k\in\N\}$ and the partial derivative $\partial_k$ will be defined in Section \ref{Notations and preliminaries}, the spaces $W^{1,p}(G^{-1}(-\infty,0),\nu)$ and $W^{1,p}(G^{-1}(-\infty,0),\nu;H)$ will be defined in section \ref{Sobolev spaces on admissible sets}, the trace operator $\trace_{G^{-1}(0)}$ will be introduced in section \ref{Traces} and the divergence operator $\diver_\nu$ will be studied in Section \ref{Divergence operator}.

Weighted Gaussian measures are an important part of infinite dimensional analysis. They arise in many problems in stochastic analysis (see for example \cite{DPZ02} and \cite{DP04}) and in obstacle problems theory (see for example \cite{Zam17}). We used the theory of Sobolev spaces with respect to a log-concave weighted Gaussian measure in order to prove maximal Sobolev regularity estimates (see \cite{CF16} and \cite{CF17}) and to characterize the domain of elliptic operators in Wiener spaces (see \cite{ACF17}).

The paper is organized in the following way: in Section \ref{Notations and preliminaries} we will introduce the basic notations that we will use throughout the paper. 

In Section \ref{Weighted Sobolev spaces} we will introduce the Sobolev spaces $W^{1,p}(X,\nu)$ and study their basic properties. We will show that the following integration by parts formula holds:
\begin{gather*}
\int_X\partial_hf(x)d\nu(x)=\int_Xf(x)(\hat{h}(x)-\partial_h\log w(x))d\nu(x)\qquad \text{for every }f\in\fcon^\infty_b(X)\text{ and }h\in H,
\end{gather*}
where $\fcon_b^\infty(X)$ denotes the vector space of the \emph{smooth cylindrical functions} (the functions of the form $f(x)=\varphi(l_1(x),\ldots,l_n(x))$, for some $\varphi\in\con^{\infty}_b(\R^n)$, $l_1,\ldots,l_n\in X^*$ and $n\in\N$), while $\hat{h}$ and $\partial_h$ will be defined in \ref{Notations and preliminaries}.
Some of the results of this section can be found in \cite{AM88}, where a more general class of weights were consider on the space $\con[0,1]$.

In Section \ref{Divergence operator} we will introduce the divergence operator $\diver_\nu$ as minus the formal adjoint of the gradient operator along $H$ and investigate some of its basic properties. In Proposition \ref{divergence for W12} we will prove that under suitable hypotheses on the weight (namely (\ref{condition L2 divergence})) then $W^{1,2}(X,\nu;H)$ is contained in the domain of $\diver_\nu$ and $\diver_\nu \Phi\in\elle^2(X,\nu)$ for every $\Phi\in W^{1,2}(X,\nu;H)$. Example \ref{Exaple divergence} shows that the conditions in Proposition \ref{divergence for W12} are only sufficient to ensure that the domain of $\diver_\nu$ is not empty. Finally Proposition \ref{divergence with mu} shows that if $\Phi\in W^{1,q}(X,\mu;H)$ (the Gaussian Sobolev spaces) then $\diver_\nu\Phi$ belongs to $\elle^p(X,\nu)$ for some reasonable values of $p$. Furthermore an explicit formula for the calculation of $\diver_\nu$ is given in (\ref{Formula divergenza}).

In Section \ref{Sobolev spaces on admissible sets} we will introduce the Sobolev spaces on sets $G^{-1}(-\infty,0)$ in a similar way as in Section \ref{Weighted Sobolev spaces}, where $G$ is a function satisfying Hypothesis \ref{ipotesi2}. 

In Section \ref{Traces} we will introduce the trace $\trace_{G^{-1}(0)}f$ of a Sobolev function $f$ in $W^{1,p}(G^{-1}(-\infty,0),\nu)$ on $G^{-1}(0)$, where $G$ is a function satisfying Hypothesis \ref{ipotesi2}. In this section we give the proof of our main result (Theorem \ref{divergence theorem with traces}) and we show in Proposition \ref{trace and precise version}, that if $\ol{\varphi}$ is a $(1,p)$-precise version (see Section \ref{Notations and preliminaries}) of an element $\varphi$ in $W^{1,p}(G^{-1}(-\infty,0),\nu)$, then
\[\trace_{G^{-1}(0)}\varphi=\ol{\varphi}_{|_{G^{-1}(0)}}\qquad\rho\text{-a.e.}\]
where $\rho$ is the Feyel--de La Pradelle Hausdorff--Gauss surface measure introduced in \cite{Fey01} (see also Section \ref{Sobolev spaces on admissible sets} for more informations).

In Section \ref{Examples} we will show how our results can be applied to explicit examples. The chosen hypersurface will be the unit sphere and the hyperplanes. In Example \ref{esempio spazio hilbert} we study the weights $w_\lambda(x)=\operatorname{exp}(\lambda\norm{x}_X^2)$, where $X$ is a separable Hilbert space and $\lambda$ is a real number. In example \ref{esempio ell2} we study the weights $w_q(x)=\operatorname{exp}(\norm{x}_q)$ in $\ell_2$, where $q>1$ and
\[\norm{x}_q=\pa{\sum_{i=1}^{+\infty}\abs{x(i)}^q}^{\frac{1}{q}}.\]
We show that every $w_q$ satisfies Hypothesis \ref{ipotesi1} (note that $\norm{\cdot}_q$ is not continuous if $1<q<2$). Eventually in Example \ref{esempio C0} we study the weight $w(f)=\operatorname{exp}(\norm{f}_\infty)$ in the space $\con[0,1]$. In this case we will consider two type of surfaces: the hyperplanes and $\set{f\in\con[0,1]\tc \norm{f}_2=1}$. All the examples are concluded by some observations about the continuity of the trace operator from $W^{1,p}(G^{-1}(-\infty,0),\nu)$ to $\elle^q(G^{-1}(0),w\rho)$ for $q\in[1,p]$, where $\rho$ is the Feyel--de La Pradelle Hausdorff--Gauss surface measure.

\section{Notation and preliminaries} \label{Notations and preliminaries}

We will denote by $X^*$ the topological dual of $X$. We recall that $X^*\subseteq\elle^2(X,\mu)$. The linear operator $R_\mu:X^*\ra (X^*)'$
\begin{gather}\label{operatore di covariaza}
R_\mu x^*(y^*)=\int_X x^*(x)y^*(x)d\mu(x)
\end{gather}
is called the covariance operator of $\mu$. Since $X$ is separable, then it is actually possible to prove that $R_\mu:X^*\ra X$ (see \cite[Theorem 3.2.3]{Bog98}). We denote by $X^*_\mu$ the closure of $X^*$ in $\elle^2(X,\mu)$. The covariance operator $R_\mu$ can be extended by continuity to the space $X^*_\mu$, still by \eqref{operatore di covariaza}. By \cite[Lemma 2.4.1]{Bog98} for every $h\in H$ there exists a unique $g\in X^*_\mu$ with $h= R_\mu g$, in this case we set
\begin{gather*}\label{definizione hat}
\hat{h}:=g.
\end{gather*}

Throughout the paper we fix an orthonormal basis $\set{e_i}_{i\in\N}$ of $H$ such that $\hat{e}_i$ belongs to $X^*$, for every $i\in\N$. Such basis exists by \cite[Corollary 3.2.8(ii)]{Bog98}.

\subsection{Differentiability along $H$}

We say that a function $f:X\ra\R$ is \emph{differentiable along $H$ at $x$} if there exists $v\in H$ such that
\[\lim_{t\ra 0}\frac{f(x+th)-f(x)}{t}=\gen{v,h}_H,\]
uniformly with respect to $h\in H$, with $\abs{h}_H=1$. In this case, the vector $v\in H$ is unique and we set $\nabla_H f(x):=v$. Moreover, for every $k\in\N$ the derivative of $f$ in the direction of $e_k$ exists and it is given by
\begin{gather*}
\partial_k f(x):=\lim_{t\ra 0}\frac{f(x+te_k)-f(x)}{t}=\gen{\nabla_H f(x),e_k}_H.
\end{gather*}

We denote by $\mathcal{H}_2$ the space of the Hilbert--Schmidt operators in $H$, that is the space of the bounded linear operators $A:H\ra H$ such that $\norm{A}_{\mathcal{H}_2}^2=\sum_{i}\abs{Ae_i}^2_H$ is finite (see \cite{DU77}).
We say that a function $f:X\ra\R$ is \emph{twice differentiable along $H$ at $x$} if it is differentiable along $H$ at $x$ and there exists $A\in\mathcal{H}_2$ such that
\[H\text{-}\lim_{t\ra 0}\frac{\nabla_Hf(x+th)-\nabla_Hf(x)}{t}=A h,\]
uniformly with respect to $h\in H$, with $\abs{h}_H=1$. In this case the operator $A$ is unique and we set $\nabla_H^2 f(x):=A$. Moreover, for every $i,j\in\N$ we set
\begin{gather*}
\partial_{ij} f(x):=\lim_{t\ra 0}\frac{\partial_jf(x+te_i)-\partial_jf(x)}{t}=\langle\nabla_H^2 f(x)e_j,e_i\rangle_H.
\end{gather*}

\subsection{Special classes of functions}\label{Special classes of functions}

For $k\in\N\cup\set{\infty}$, we denote by $\fcon^k(X)$ ($\fcon_b^k(X)$ respectively) the space of the cylindrical function of the type
\(f(x)=\varphi(x^*_1(x),\ldots,x^*_n(x))\)
where $\varphi\in\con^{k}(\R^n)$ ($\varphi\in\con^{k}_b(\R^n)$, respectively) and $x^*_1,\ldots,x^*_n\in X^*$, for some $n\in\N$. We remark that by a classical argument it is possible to prove that $\fcon^\infty_b(X)$ is dense in $\elle^p(X,\nu)$ for all $p\geq 1$ (see \cite{DPL14}). We recall that if $f\in \fcon^2(X)$, then $\partial_{ij}f(x)=\partial_{ji}f(x)$ for every $i,j\in\N$ and $x\in X$.

Let $L_p$ be the infinitesimal generator of the \emph{Ornstein--Uhlenbeck semigroup} $\T(t)$ in $\elle^p(X,\mu)$, where
\[T(t)f(x):=\int_Xf\pa{e^{-t}x+(1-e^{-2t})^{\frac{1}{2}}y}d\mu(y)\qquad\text{ for }t>0.\]
For $k=1,2,3$, we define the \emph{$C_{k,p}$-capacity} of an open set $A\subseteq X$ as
\[C_{k,p}(A):=\inf\set{\norm{f}_{\elle^p(X,\mu)}\tc (I-L_p)^{-\frac{k}{2}}f\geq 1\ \mu\text{-a.e. in }A}.\]
For a general Borel set $B\subseteq X$ we let $C_{k,p}(B)=\inf\set{C_{k,p}(A)\tc B\subseteq A\text{ open}}$. By $f\in W^{k,p}(X,\mu)$ we mean an equivalence class of functions and we call every element ``version''. For any $f\in W^{k,p}(X,\mu)$ there exists a version $\ol f$ of $f$ which is Borel measurable and \emph{$C_{k,p}$-quasicontinuous}, i.e. for every $\eps>0$ there exists an open set $A\subseteq X$ such that $C_{k,p}(A)\leq \eps$ and $\ol{f}_{|_{X\ssm A}}$ is continuous. Furthermore, for every $r>0$
\[C_{k,p}\pa{\set{x\in X\tc \abs{\ol{f}(x)}>r}}\leq\frac{1}{r}\norm{(I-L_p)^{-\frac{k}{2}}\ol{f}}_{\elle^p(X,\mu)}.\]
See \cite[Theorem 5.9.6]{Bog98}. Such a version is called a \emph{$(k,p)$-precise version of $f$}. Two precise versions of the same $f$ coincide outside sets with null $C_{k,p}$-capacity. All our results will be independent on our choice of a precise version of $G$ in Hypothesis \ref{ipotesi2}. With obvious modification the same definition can be adapted to functions belonging to $W^{k,p}(X,\mu;H)$ and $W^{k,p}(X,\mu;\mathcal{H}_2)$.

\subsection{Sobolev spaces}
The Gaussian Sobolev spaces $W^{1,p}(X,\mu)$ and $W^{2,p}(X,\mu)$, with $p\geq 1$, are the completions of the \emph{smooth cylindrical functions} $\fcon_b^\infty(X)$ in the norms
\begin{gather*}
\norm{f}_{W^{1,p}(X,\mu)}:=\norm{f}_{\elle^p(X,\mu)}+\pa{\int_X\abs{\nabla_H f(x)}_H^pd\mu(x)}^{\frac{1}{p}};\\
\norm{f}_{W^{2,p}(X,\mu)}:=\norm{f}_{W^{1,p}(X,\mu)}+\pa{\int_X\norm{\nabla_H^2 f(x)}^p_{\mathcal{H}_2}d\mu(x)}^{\frac{1}{p}}.
\end{gather*}
Such spaces can be identified with subspaces of $\elle^p(X,\mu)$ and the (generalized) gradient and Hessian along $H$, $\nabla_H f$ and $\nabla_H^2 f$, are well defined and belong to $\elle^p(X,\mu;H)$ and $\elle^p(X,\mu;\mathcal{H}_2)$, respectively. The spaces $W^{1,p}(X,\mu;H)$ are defined in a similar way, replacing smooth cylindrical functions with $H$-valued smooth cylindrical functions (i.e. the linear span of the functions $x\mapsto f(x)h$, where $f$ is a smooth cylindrical function and $h\in H$). For more information see \cite[Section 5.2]{Bog98}.

We remind the reader that the following integration by parts formula holds for every $f\in W^{1,p}(X,\mu)$, with $p>1$, and $h\in H$
\begin{gather}\label{Gaussian int by part}
\int_X\partial_hf(x)d\mu(x)=\int_Xf(x)\hat{h}(x)d\mu(x).
\end{gather}

\subsection{Surface measures} 
For a comprehensive treatment of surface measures in infinite dimensional Banach spaces with Gaussian measures we refer to \cite{FP91}, \cite{Fey01} and \cite{CL14}. We recall the definition of the Feyel--de La Pradelle Hausdorff--Gauss surface measure. If $m\geq 2$ and $F=\R^m$ equipped with a norm $\norm{\cdot}_F$, we define
\[d\theta^F(x)=\frac{1}{(2\pi)^{\frac{m}{2}}}e^{-\frac{\norm{x}_F^2}{2}}dH_{m-1}(x),\]
where $H_{m-1}$ is the spherical $(m-1)$-dimensional Hausdorff measure in $\R^m$, i.e.
\[H_{m-1}(A)=\lim_{\delta\ra 0}\inf\set{\sum_{n\in\N}w_{m-1}r_n^{m-1}\tc A\subseteq \bigcup_{n\in\N}B(x_n,r_n),\ r_n<\delta,\ \text{for every }n\in\N},\]
where $w_{m-1}=\pi^{\frac{m-1}{2}}(\Gamma(\frac{m+1}{2}))^{-1}$. For every $m$-dimensional $F\subseteq H$ we consider the orthogonal projection (along $H$) on $F$:
\[x\mapsto\sum_{n=1}^m\gen{x,f_n}_Hf_n\qquad x\in H\]
where $\set{f_n}_{n=1}^m$ is an orthonormal basis of $F$. There exists a $\mu$-measurable projection $\pi^F$ on $F$, defined in the whole $X$, that extends it (see \cite[Theorem 2.10.11]{Bog98}). We denote by $\tilde{F}:=\ker\pi^F$ and by $\mu_{\tilde{F}}$ the image of the measure $\mu$ on $\tilde{F}$ through $I-\pi^F$. Finally we denote by $\mu_F$ the image of the measure $\mu$ on $F$ through $\pi^F$, which is the standard Gaussian measure on $\R^m$ if we identify $F$ with $\R^m$. Let $A\subseteq X$ be a Borel set and identify $F$ with $\R^m$, we set
\[\rho^F(A):=\int_{\ker \pi^F}\theta^F(A_x)d\mu_{\tilde{F}}(x),\]
where $A_x=\set{y\in F\tc x+y\in A}$. The map $F\mapsto\rho^F(A)$ is well defined and increasing, namely if $F_1\subseteq F_2$ are finite dimensional subspaces of $H$, then $\rho^{F_1}(A)\leq \rho^{F_2}(A)$ (see \cite[Lemma 3.1]{AMP10} and \cite[Proposition 3.2]{Fey01}). The Feyel--de La Pradelle Hausdorff--Gauss surface measure is defined by
\[\rho(A)=\sup\set{\rho^F(A)\tc F\subseteq H,\ F\text{ is a finite dimensional subspace}}.\]

\section{Weighted Sobolev spaces}\label{Weighted Sobolev spaces}

We want to define the Sobolev space $W^{1,p}(X,\nu)$ as the domain of the closure of the gradient operator along $H$. A natural procedure is to prove an integration by parts formula.

\begin{lemma}\label{int by part}
Let $f\in\fcon^{\infty}_b(X)$ and $h\in H$. The following formula holds:
\begin{gather}\label{int}
\int_X\partial_hf(x)d\nu(x)=\int_Xf(x)(\hat{h}(x)-\partial_h\log w(x))d\nu(x).
\end{gather}
\end{lemma}

\begin{proof}
Using the integration by parts formula for the Gaussian measure  (\ref{Gaussian int by part}) we get
\begin{gather*}
\int_X\partial_hf(x)d\nu(x)=\int_X\partial_hf(x)w(x)d\mu(x)=\int_X\partial_h(f(x)w(x))d\mu(x)-\int_Xf(x)
\partial_hw(x)d\mu(x)=\\
=\int_X\hat{h}(x)f(x)w(x)d\mu(x)-\int_Xf(x)\frac{\partial_hw(x)}{w(x)}w(x)d\mu(x)=\int_Xf(x)(\hat{h}(x)-\partial_h\log w(x))d\nu(x).
\end{gather*}
\end{proof}
We are now ready to prove the closability of the gradient operator along $H$.

\begin{pro}\label{closure gradient}
The operator $\nabla_H:\fcon^{\infty}_b(X)\ra L^p(X,\nu;H)$ is closable in $L^p(X,\nu)$, whenever $p\geq\frac{t}{t-s'}$.
\end{pro}

\begin{proof}
Let $(f_k)_{k\in\N}\subseteq\fcon^{\infty}_b(X)$ be such that
\begin{align*}
&\lim_{k\ra+\infty}f_k=0 \qquad \text{ in } L^p(X,\nu);\\
&\lim_{k\ra+\infty}\nabla_H f_k=\Phi \qquad \text{ in } L^p(X,\nu;H).
\end{align*}
We want to prove that $\Phi(x)=0$, $\nu$-a.e. To this aim we will show that
\[\int_X\gen{\Phi(x),e_n}_Hu(x)d\nu(x)=0\]
for every $n\in\N$ and $u\in\fcon^{\infty}_b(X)$. Recall that if $f,g\in\fcon^{\infty}_b(X)$, then $fg\in\fcon_b^{\infty}(X)$. By the integration by parts formula (Lemma \ref{int by part}), we get
\begin{gather}\label{1}
\int_X\partial_nf_k(x)u(x)d\nu(x)=\int_Xf_k(x)(\hat{e}_n(x)-\partial_n\log w(x))u(x)d\nu(x)-\int_Xf_k(x)\partial_nu(x)d\nu(x).
\end{gather}
By the H\"older inequality we get
\begin{gather*}
\int_X\abs{\partial_nf_k(x)-\gen{\Phi(x),e_n}}
\abs{u(x)}d\nu(x)\leq\\ 
\leq
\pa{\int_X\abs{\partial_nf_k(x)-\gen{\Phi(x),e_n}}^pd\nu(x)}^{\frac{1}{p}}\pa{\int_X\abs{u(x)}^{p'}d\nu(x)}^{\frac{1}{p'}}\xra{k\ra+\infty}0.
\end{gather*}
Furthermore
\begin{gather*}
\int_X\abs{f_k(x)\partial_nu(x)}d\nu(x)\leq \pa{\int_X\abs{f_k(x)}^pd\nu(x)}^{\frac{1}{p}}\pa{\int_X\abs{\partial_n u(x)}^{p'}d\nu(x)}^{\frac{1}{p'}}\xra{k\ra+\infty}0.
\end{gather*}
Lastly
\begin{gather*}
\int_X\abs{f_k}\abs{\hat{e}_n-\partial_n\log w}\abs{u}d\nu
\leq\norm{u}_\infty\int_X\abs{f_k}
\abs{\hat{e}_n-\partial_n\log w}d\nu\leq_{(1)}\\
\leq_{(1)}\norm{u}_\infty
\pa{\int_X\abs{f_k}^pd\nu}^{\frac{1}{p}}
\pa{\int_X\abs{\hat{e}_n-\partial_n\log w}^{p'}wd\mu}^{\frac{1}{p'}}\leq_{(2)}\\
\leq_{(2)}\norm{u}_\infty
\pa{\int_X\abs{f_k}^pd\nu}^{\frac{1}{p}}
\pa{\int_X w^sd\mu}^{\frac{1}{s}}
\pa{\int_X\abs{\hat{e}_n-\partial_n\log w}^{p's'}d\mu}^{\frac{1}{p's'}}\xra{k\ra+\infty}0;
\end{gather*}
where both $(1)$ and $(2)$ follow applying the H\"older inequality. Note that the last integral is finite whenever $p's'\leq t$, i.e. $p\geq \frac{t}{t-s'}$ which is an assumption. Letting $k\ra+\infty$ in (\ref{1}) we get
\[\int_X\gen{\Phi(x),e_n}u(x)d\nu(x)=0\]
for every $n\in\N$ and $u\in\fcon_b^{\infty}(X)$.
\end{proof}

\begin{defn}[Weighted Sobolev spaces]\label{Weightes Sobolev space definition}
Let $p\geq\frac{t}{t-s'}$. We denote by $W^{1,p}(X,\nu)$ the domain of the closure of $\nabla_H$ (which we still denote by the symbol $\nabla_H$) in $L^p(X,\nu)$. It is a Banach space with the graph norm
\[\norm{f}_{W^{1,p}(X,\nu)}=\pa{\int_X\abs{f(x)}^pd\nu(x)}^{\frac{1}{p}}+\pa{\int_X\abs{\nabla_Hf(x)}_H^pd\nu(x)}^{\frac{1}{p}.}\]
In the same way we define $W^{1,p}(X,\nu;H)$ using $H$-valued smooth cylindrical functions. Furthermore Lemma \ref{int by part} holds for every $f\in W^{1,p}(X,\nu)$.
\end{defn}
Using a standart argument it is possible to prove that \eqref{int} holds for every $f\in W^{1,p}(X,\nu)$, with $p\geq\frac{t}{t-s'}$, and $h\in H$. The first thing we want to show is the following result about the coincidence of the closure of the gradient operator along $H$ in $\elle^p(X,\nu)$ and $\elle^q(X,\mu)$.

\begin{pro}\label{coincidenza gradienti}
In this proposition we will denote by $\nabla_H^{\mu,q}$ and $\nabla_H^{\nu,p}$ the closure of the gradient operator along $H$ in $\elle^q(X,\mu)$ and $\elle^{p}(X,\nu)$, respectively. Let $p\geq\frac{t}{t-s'}$ and $(f_n)_{n\in\N}\subseteq \fcon^\infty_b(X)$. If there exists $f\in W^{1,q}(X,\mu)$, for some $q> ps'$, such that
\begin{gather*}
\lim_{n\ra+\infty} f_n=f\qquad\text{in }\elle^q(X,\mu);\\
\lim_{n\ra+\infty} \nabla_H f_n=\nabla_H^{\mu,q}f\qquad\text{in }\elle^q(X,\mu;H),
\end{gather*}
then
\begin{gather*}
\lim_{n\ra+\infty} f_n=f\qquad\text{in }\elle^p(X,\nu);\\
\lim_{n\ra+\infty} \nabla_H f_n=\nabla_H^{\nu,p}f\qquad\text{in }\elle^p(X,\nu;H).
\end{gather*}
Furthermore $\nabla_H^{\nu,p}f=\nabla_H^{\mu,q}f$ $\mu$-a.e.
\end{pro}

\begin{proof}
By the H\"older inequality we get
\begin{gather*}
\int_X\abs{f_n-f}^pd\nu\leq\pa{\int_Xw^sd\mu}^{\frac{1}{s}}\pa{\int_X\abs{f_n-f}^{ps'}d\mu}^{\frac{1}{s'}}\leq \\
\leq\pa{\int_Xw^sd\mu}^{\frac{1}{s}}\pa{\int_X\abs{f_n-f}^{q}d\mu}^{\frac{p}{q}}\xra{n\ra+\infty}0.
\end{gather*}
So $\lim_{n\ra+\infty} f_n=f$ in $\elle^p(X,\nu)$. In the same way
\begin{gather*}
\int_X\abs{\nabla_H f_n-\nabla_H^{\mu,q}f}_H^pd\nu\leq\pa{\int_Xw^sd\mu}^{\frac{1}{s}}\pa{\int_X\abs{\nabla_H f_n-\nabla_H^{\mu,q}f}_H^{ps'}d\mu}^{\frac{1}{s'}}\leq \\
\leq\pa{\int_Xw^sd\mu}^{\frac{1}{s}}\pa{\int_X\abs{\nabla_H f_n-\nabla_H^{\mu,q}f}_H^{q}d\mu}^{\frac{p}{q}}\xra{n\ra+\infty}0.
\end{gather*}
So $\lim_{n\ra+\infty} \nabla_H f_n=\nabla_H^{\mu,q}f$ in $\elle^p(X,\nu;H)$. The furthermore part is now obvious.
\end{proof}

It is important to observe some basic properties of the space $W^{1,p}(X,\nu)$.

\begin{pro}\label{proprerties of sobolev space}
Let $p\geq\frac{t}{t-s'}$. The following holds:
\begin{enumerate}
\item $W^{1,p}(X,\nu)$ is reflexive; \label{rifelssivita' in nu}

\item for every $q\in (ps',+\infty)$, $W^{1,q}(X,\mu)\hookrightarrow W^{1,p}(X,\nu)$. In particular if $G$ satisfies Hypothesis \ref{ipotesi2} then $G\in W^{1,b}(X,\nu)$ for every $b\geq\frac{t}{t-s'}$. The same is true for the spaces $W^{1,q}(X,\mu;H)$ and $W^{1,p}(X,\nu;H)$; \label{W(mu) in W(nu)}

\item if $\int_X w^{r(r-1)^{-1}}d\mu<+\infty$ for some $r\in(0,1)$, then $W^{1,p}(X,\nu)\hookrightarrow W^{1,pr}(X,\mu)$.  The same is true for the space $W^{1,p}(X,\nu;H)$; \label{W(nu) in W(mu)}

\item let $(\varphi_n)_{n\in\N}\subseteq W^{1,p}(X,\nu)$. If $\varphi_n$ converges pointwise $\nu$-a.e. to $\varphi$ and 
\[\sup_{n\in\N}\norm{\varphi_n}_{W^{1,p}(X,\nu)}<+\infty,\]
then $\varphi\in W^{1,p}(X,\nu)$; \label{Banach Saks}

\item let $q\geq\frac{t}{t-s'}$, $\varphi\in W^{1,p}(X,\nu)$ and $\psi\in W^{1,q}(X,\nu)$. If $\frac{pq}{p+q}\geq\frac{t}{t-s'}$ then
\[\varphi\psi\in W^{1,r}(X,\nu)\qquad\text{for every }r\in \left[\frac{t}{t-s'},\frac{pq}{p+q}\right],\]
and $\nabla_H(\varphi\psi)=\varphi\nabla_H\psi+\psi\nabla_H\varphi$. \label{prodotto funzioni in W}
\end{enumerate}
\end{pro}

\begin{proof}
\begin{enumerate}
\item  It is easily seen that $\elle^p(X,\nu)\times\elle^p(X,\nu;H)$ is a reflexive Banach space when endowed with the norm
\[\norm{(f,\Phi)}=\norm{f}_{\elle^p(X,\nu)}+\norm{\Phi}_{\elle^p(X,\nu;H)}.\]
This is due to the fact that the weakly closed and norm bounded sets of both $\elle^p(X,\nu)$ and $\elle^p(X,\nu;H)$ are weakly compact (see \cite[Corollary IV.1.2]{DU77}). Thus the set 
\[\set{(f,\Phi)\in\elle^p(X,\nu)\times\elle^p(X,\nu;H)\tc \norm{(f,\Phi)}\leq 1}\]
is weakly compact and so $\elle^p(X,\nu)\times\elle^p(X,\nu;H)$ is reflexive.
The operator $T:W^{1,p}(X,\nu)\ra \elle^p(X,\nu)\times\elle^p(X,\nu;H)$ defined as
\[T(f)=(f,\nabla_H f),\] 
is an isometric embedding, which implies that the range of $T$ is closed in $\elle^p(X,\nu)\times\elle^p(X,\nu;H)$. Thus $T(W^{1,p}(X,\nu))$ is reflexive, being a closed subspace of a reflexive space. So $W^{1,p}(X,\nu)$ is reflexive too, being isometric to a reflexive space.

\item Let $f\in W^{1,q}(X,\mu)$. Using the H\"older inequality we get
\begin{gather*}
\int_X\abs{f}^pd\nu=\int_X\abs{f}^pwd\mu\leq
\pa{\int_X\abs{f}^{ps'}d\mu}^{\frac{1}{s'}}
\pa{\int_Xw^sd\mu}^{\frac{1}{s}},
\end{gather*}
and the right hand side is finite whenever $ps'< q$. Using Proposition \ref{coincidenza gradienti}, the same inequality holds for $\nabla_H f$. Thus the statement follows.

\item Let $f\in W^{1,p}(X,\nu)$. Using the H\"older inequality we get
\begin{gather*}
\int_X\abs{f}^{pr}d\mu=\int_X\frac{\abs{f}^{pr}}{w}d\nu\leq\pa{\int_X\abs{f}^pd\nu}^r\pa{\int_Xw^{\frac{r}{r-1}}d\mu}^{1-r}.
\end{gather*}
Using the same argument in Proposition \ref{coincidenza gradienti} it is possible to prove the same inequality for $\nabla_H f$. Thus the statement follows.

\item The statement is a consequence of the Banach--Saks property, i.e. every bounded sequence $(\varphi_n)_{n\in\N}\subseteq\elle^p(X,\nu;H)$ has a subsequence $(\varphi_{n_k})_{k\in\N}$ such that the sequence
\[\pa{\frac{\varphi_{n_1}+\cdots+\varphi_{n_k}}{k}}_{k\in\N}\]
strongly converges in $\elle^p(X,\nu;H)$. The proof is the same of \cite[Lemma 5.4.4]{Bog98} with obvious adjustments. 

\item Standard calculations.
\end{enumerate}
\end{proof}

\begin{pro}\label{C1 func}
Let $p\geq\frac{t}{t-s'}$ and $\theta\in\con^1_b(\R)$. If $\varphi\in W^{1,p}(X,\nu)$, then $\theta\circ\varphi\in W^{1,p}(X,\nu)$ and
\[\nabla_H(\theta\circ\varphi)=(\theta'\circ\varphi)\nabla_H\varphi\qquad\text{ $\nu$-a.e.}\]
\end{pro}

\begin{proof}
Let $(\varphi_n)_{n\in\N}\subseteq\fcon_b^\infty(X)$ be such that $\varphi_n$ converges in $W^{1,p}(X,\nu)$ and pointwise $\nu$-a.e. to $\varphi$. Applying (\ref{Banach Saks}) of Proposition \ref{proprerties of sobolev space} to the sequence $\theta\circ\varphi_n$ we get $\theta\circ\varphi\in W^{1,p}(X,\nu)$. Furthermore we know that, for every $n\in\N$, $\theta\circ \varphi_n\in\fcon_b^1(X)$ and
\[\nabla_H(\theta\circ\varphi_n)=(\theta'\circ\varphi_n)\nabla_H\varphi_n.\]
By Proposition \ref{closure gradient} and the fact that $\theta\in\con^1_b(\R)$, we can let $n\ra+\infty$ and get the statement.
\end{proof}

\begin{pro}\label{derivate of modulus}
Let $p\geq\frac{t}{t-s'}$ and $u\in W^{1,p}(X,\nu)$. Then $\abs{u}\in W^{1,p}(X,\nu)$ and
\[\nabla_H\abs{u}=\sign(u)\nabla_H u.\]
Furthermore $\nabla_H u(x)=0$  $\nu$-a.e. in $\set{y\in X\tc u(y)=0}$.
\end{pro}

\begin{proof}
The proof is the same as in \cite[Lemma 2.7]{DPL14}, we just use Lemma \ref{int by part} instead of the classical integration by parts formula for the Gaussian measure (\ref{Gaussian int by part}).
\end{proof}

\section{Divergence operator}\label{Divergence operator}

We start this section by recalling the definition of divergence, see \cite[Section 5.8]{Bog98}. For every measurable map $\Phi:X\ra X$ and for every $f\in\fcon^\infty_b(X)$ we define
\begin{gather}\label{divergenza su dominio}
\partial_\Phi f(x)=\lim_{t\ra 0}\frac{f(x+t\Phi(x))-f(x)}{t},\qquad\qquad x\in X.
\end{gather}

\begin{defn}
Let $\Phi\in \elle^1(X,\nu;X)$ be a vector field. We say that $\Phi$ admits \emph{divergence} if there exists a function $g\in\elle^1(X,\nu)$ such that
\begin{gather}\label{condizione di divergenza in X}
\int_X\partial_{\Phi}fd\nu=-\int_X fgd\nu\qquad \text{for every }f\in\fcon_b^\infty(X);
\end{gather}
where $(\partial_{\Phi}f)(x)$ is defined in (\ref{divergenza su dominio}). If such a function exists then we set 
\[\diver_\nu \Phi:=g.\] 
Observe that, when $\diver_\nu \Phi$ exists, it is unique by the density of $\fcon^\infty_b(X)$. We denote by $D(\diver_\nu)$ the domain of $\diver_\nu$ in $\elle^1(X,\nu;X)$. Lastly observe that if $\Phi\in\elle^1(X,\nu;H)$, then $(\partial_{\Phi}f)(x)=\gen{\nabla_Hf(x),\Phi(x)}_H$. In this case (\ref{condizione di divergenza in X}) becomes
\begin{gather}\label{condizione di divergenza in H}
\int_X\gen{\nabla_Hf(x),\Phi(x)}_Hd\nu=-\int_X fgd\nu\qquad \text{for every }f\in\fcon_b^\infty(X).
\end{gather}
\end{defn}
We are now interested in conditions ensuring the non emptyness of $D(\diver_\nu)$ and $\elle^p$-integrability of $\diver_\nu v$. We start by studying the case of a vector field in $W^{1,2}(X,\nu;H)$.

\begin{pro}\label{Conti per divergenza}
Assume Hypothesis \ref{ipotesi1} holds and that $\log w\in W^{2,t}(X,\mu)$, where $t$ is the exponent fixed in Hypothesis \ref{ipotesi1}. For every $f,g\in\fcon^{\infty}_b(X)$ and $h,k\in H$
\begin{gather*}
\int_X\pa{f\hat{h}-f\partial_h\log w-\partial_h f}\pa{g\hat{k}-g\partial_k\log w-\partial_kg}d\nu=\\
=\gen{h,k}_H\int_Xfgd\nu-\int_Xfg\partial_h\partial_k\log wd\nu+\int_X\partial_kf\partial_h gd\nu.
\end{gather*}
\end{pro}

\begin{proof}
By the integration by parts formula (Lemma \ref{int by part}) we get
\begin{gather*}
\int_X\pa{f\hat{h}-f\partial_h\log w-\partial_h f}\pa{g\hat{k}-g\partial_k\log w-\partial_kg}d\nu=\\
=\int_Xf\partial_h\pa{g\hat{k}}d\nu-\int_Xf\partial_h\pa{g\partial_k\log w}d\nu-\int_Xf\partial_h\partial_k gd\nu.
\end{gather*}
Recalling that $\partial_h\hat{k}=\gen{h,k}_H$ we get
\begin{gather*}
\int_X\pa{f\hat{h}-f\partial_h\log w-\partial_h f}\pa{g\hat{k}-g\partial_k\log w-\partial_kg}d\nu=\\
=\gen{h,k}_H\int_Xfgd\nu-\int_Xfg\partial_h\partial_k\log wd\nu+\int_X\partial_kf\partial_h gd\nu.
\end{gather*}
\end{proof}

\begin{pro}\label{divergence for W12}
Let $\log w\in W^{2,t}(X,\mu)$, for some $t\geq 2s'$, and assume that there exists $C\geq 0$ such that for every $(\xi_i)_{i\in\N}\in\ell_2$ we have
\begin{gather}\label{condition L2 divergence}
\sum_{i=1}^{+\infty}\sum_{j=1}^{+\infty}(\delta_{ij}-\partial_i\partial_j\log w(x))\xi_i\xi_j\leq C\sum_{i=1}^{+\infty}\xi^2_i\qquad\mu\text{-a.e.}
\end{gather}
Then every field $\Phi\in W^{1,2}(X,\nu;H)$ has a divergence $\diver_\nu \Phi\in L^2(X,\nu)$ and for every $f\in W^{1,2}(X,\nu)$, the following equality holds:
\[\int_X\gen{\nabla_H f(x),\Phi(x)}_Hd\nu(x)=-\int_X f(x)\diver_\nu \Phi(x)d\nu(x).\]
Furthermore, if $\varphi_n=\gen{\Phi,e_n}_H$ for every $n\in\N$ where $(e_n)_{n\in\N}$ is an orthonormal basis of $H$, then
\[\diver_\nu \Phi(x)=\sum_{n=1}^{+\infty}\pa{\partial_n\varphi_n(x)+\varphi_n(x)\partial_n\log w(x)-\varphi_n(x)\hat{e}_n},\]
and $\norm{\diver_\nu \Phi}_{\elle^2(X,\nu)}\leq \max\{\sqrt{C},1\}\norm{\Phi}_{W^{1,2}(X,\nu;H)}$.
\end{pro}

\begin{proof}
Let $\Phi(x)=\sum_{i=1}^n \varphi_i(x)e_i$ with $\varphi_i\in\fcon^{\infty}_b(X)$ for every $i=1,\ldots,n$. Then by the integration by parts formula (Lemma \ref{int by part}) if $f\in \fcon^\infty_b(X)$ we have
\[\int_X\gen{\nabla_H f(x),\Phi(x)}_Hd\nu(x)=-\int_X f(x)\pa{\sum_{i=1}^{n}\pa{\partial_i\varphi_i(x)+\varphi_i(x)\partial_i\log w(x)-\varphi_i(x)\hat{e}_i}}d\nu(x).\]
Put
\[\diver_\nu \Phi(x)=\sum_{i=1}^{n}\pa{\partial_i\varphi_i(x)+\varphi_i(x)\partial_i\log w(x)-\varphi_i(x)\hat{e}_i}.\]
By Proposition \ref{Conti per divergenza}
\begin{gather}
\notag\int_X\pa{\diver_\nu \Phi}^2d\nu=\sum_{i=1}^n\int_X\abs{\varphi_i}^2d\nu+\sum_{i=1}^n\sum_{j=1}^n\int_X\partial_j
\varphi_i\partial_i \varphi_jd\nu-\sum_{i=1}^n\sum_{j=1}^n\int_X \varphi_i\varphi_j\partial_i\partial_j\log wd\nu=\\
\label{formula divergenza}=\int_X\sum_{i=1}^n\sum_{j=1}^n(\delta_{ij}-\partial_i\partial_j\log w)\varphi_i\varphi_jd\nu+\int_X\text{trace}_H((\nabla_H \Phi(x))^2)d\nu(x)\leq\\
\notag\leq C\norm{\Phi}^2_{L^2(X,\nu;H)}+\int_X\norm{\nabla_H \Phi(x)}_{\mathcal{H}_2}^2d\nu(x)\leq\max\set{C,1}\norm{\Phi}^2_{W^{1,2}(X,\nu;H)}.
\end{gather}
Let $(\Phi^n)_{n\in\N}$ be a sequence of fields as above which converges to $\Phi$ in $W^{1,2}(X,\nu;H)$. The sequence $(\diver_\nu \Phi^n)$ is Cauchy in $\elle^2(X,\nu)$ and, hence it converges to some element, which is the candidate to be $\diver_\nu \Phi$. By the integration by parts formula (Lemma \ref{int by part}) it is easily seen that such an element satifies (\ref{condizione di divergenza in H}). Finally, the fact that (\ref{condizione di divergenza in H}) holds also for every $f\in W^{1,2}(X,\nu)$, follows by a standard approximation argument.
\end{proof}
Observe that (\ref{condition L2 divergence}) is satisfied whenever $\nabla_H^2\log w$ is bounded from below $\mu$-a.e., in particular when $\log w$ is a convex function.

The following example show that (\ref{condition L2 divergence}) is not necessary for the domain of the divergence to be not empty.

\begin{eg}\label{Exaple divergence}
Let $X$ be a separable Hilbert space, with norm $\norm{\cdot}_X$ and inner product $(\cdot,\cdot)_X$, endowed with a nondegenerate centered Gaussian measure $\mu$, with covariance $Q$. We fix an orthonormal basis $(v_n)_{n\in\N}$ of $X$ of eigenvectors of $Q$, i.e. $Qv_k=\lambda_k v_k$, and the corresponding orthonormal basis of $H=Q^{\frac{1}{2}}(X)$ is $(e_n:=\sqrt{\lambda_n}v_n)_{n\in\N}$. Recall that $\hat{e}_n(x)=\frac{(x,v_n)_X}{\sqrt{\lambda_n}}$. In this example we take $\lambda_n=4^{-n}$.

Let $w(x):=(x,x)^2_X=\sum_{n=1}^{+\infty}(x,v_n)_X^2=\sum_{n=1}^{+\infty}\lambda_n^{-1}(x,e_n)_X^2$. Easy calculations give
\begin{gather*}
\partial_iw(x)=4(x,x)_X(x,e_i)_X,\quad \partial_i\log w(x)=4\frac{(x,e_i)_X}{(x,x)_X},\\ \partial_i\partial_j\log w(x)=4\pa{\frac{(e_i,e_j)_X}{(x,x)_X}-\frac{2(x,e_i)_X(x,e_j)_X}{(x,x)^2_X}}.
\end{gather*}
We get $w\in W^{1,s}(X,\mu)$ and $\log w\in W^{2,t}(X,\mu)$ for every $t,s>1$. Therefore $W^{1,p}(X,\nu)$ is well defined for every $p>1$ (see Definition \ref{Weightes Sobolev space definition}).

We want to show that $w$ does not satisfy (\ref{condition L2 divergence}). For every $(\xi_i)_{i\in\N}\in\ell_2$ we have
\begin{gather}\label{in esempio divergenza}
\sum_{i=1}^{+\infty}\sum_{j=1}^{+\infty}(\delta_{ij}-\partial_i\partial_j\log w(x))\xi_i\xi_j=\sum_{i=1}^{+\infty}\sum_{j=1}^{+\infty}\pa{\delta_{ij}-4\frac{(e_i,e_j)_X}{(x,x)_X}+8\frac{(x,e_i)_X(x,e_j)_X}{(x,x)_X^2}}\xi_i\xi_j=\\ \notag
=\sum_{i=1}^{+\infty}\sum_{j=1}^{+\infty}\pa{\delta_{ij}-4\delta_{ij}\frac{\sqrt{\lambda_i\lambda_j}}{(x,x)_X}+8\frac{\sqrt{\lambda_i\lambda_j}(x,v_i)_X(x,v_j)_X}{(x,x)_X^2}}\xi_i\xi_j.
\end{gather}
By contradiction assume that $C\geq 0$ exists such that (\ref{condition L2 divergence}) holds. Choosing $\xi_i=(x,v_i)$ for every fixed $x\in X$, we get
\begin{gather*}
\sum_{i=1}^{+\infty}\sum_{j=1}^{+\infty}(\delta_{ij}-\partial_i\partial_j\log w(x))\xi_i\xi_j=\\
=\sum_{i=1}^{+\infty}(x,v_i)_X^2-\frac{4}{(x,x)_X}\sum_{i=1}^{+\infty}\frac{(x,v_i)_X^2}{4^i}+\frac{8}{(x,x)_X^2}\pa{\sum_{i=1}^{+\infty}\frac{(x,v_i)_X^2}{2^i}}^2\leq C\sum_{i=1}^{+\infty}(x,v_i)_X^2\qquad\mu\text{-a.e.}.
\end{gather*}
Let $A=\set{x\in X\tc \norm{x}_X^2<\sqrt{2^{-1}}\sum_{i=1}^{+\infty}4^{-i}(x,v_i)^2_X}$ and $B_r=\set{x\in X\tc \norm{x}_X<r}$ for $0<r<1$. Observe that $A\cap B_r$ are open non-empty subsets of $X$ for every $r\in(0,1)$, this means that $\mu(A\cap B_r)>0$.
Fix $x\in A\cap B_r$ for some $r\in(0,1)$, then
\begin{gather}\notag
(C-1)\sum_{i=1}^{+\infty}(x,v_i)_X^2+\frac{4}{(x,x)_X}\sum_{i=1}^{+\infty}\frac{(x,v_i)_X^2}{4^i}-\frac{8}{(x,x)_X^2}\pa{\sum_{i=1}^{+\infty}\frac{(x,v_i)_X^2}{2^i}}^2\leq\\ \notag
\leq (C-1)\sum_{i=1}^{+\infty}(x,v_i)_X^2+\frac{4}{(x,x)_X}\sum_{i=1}^{+\infty}\frac{(x,v_i)_X^2}{4^i}-\frac{8}{(x,x)_X^2}\pa{\sum_{i=1}^{+\infty}\frac{(x,v_i)_X^2}{4^i}}^2\leq\\ \label{controesempio divergenza 2}
\leq (C-1)\norm{x}^2_X+4-2\norm{x}_X^4\frac{8}{\norm{x}_X^4}=(C-1)\norm{x}^2_X-12
\end{gather}
Observe that if $C>1$ and $x\in A\cap B_{(6(C-1)^{-1})^{\frac{1}{2}}}$, then (\ref{controesempio divergenza 2})$\leq-6$ $\mu$-a.e., a contradiction. Thus our example does not satisfy (\ref{condition L2 divergence}).

Observe that if $\Phi\in W^{1,2\alpha}(X,\mu;H)$ with $\alpha>1$, then $\Phi\in D(\diver_\nu)$. Indeed let $\varphi_n=\gen{\Phi,e_n}_{H}$ for every $n\in\N$, then by(\ref{in esempio divergenza})
\begin{gather*}
\sum_{i=1}^{+\infty}\sum_{j=1}^{+\infty}(\delta_{ij}-\partial_i\partial_j\log w(x))\varphi_i\varphi_j=\\
=\sum_{i=1}^{+\infty}\sum_{j=1}^{+\infty}\pa{\delta_{ij}-4\delta_{ij}\frac{\sqrt{\lambda_i\lambda_j}}{(x,x)_X}+8\frac{\sqrt{\lambda_i\lambda_j}(x,v_i)_X(x,v_j)_X}{(x,x)_X^2}}\varphi_i\varphi_j=\\
=\sum_{i=1}^{+\infty}\varphi_i^2-4\sum_{i=1}^{+\infty}\frac{\lambda_i\varphi_i^2}{(x,x)_X}+\frac{8}{(x,x)_X^2}\pa{\sum_{i=1}^{+\infty}\sqrt{\lambda_i}\varphi_i(x,v_i)_X}^2\leq \\
\leq\abs{\Phi}^2_H+8\frac{\pa{x,\sum_{i=1}^{+\infty}\varphi_ie_i}^2_X}{(x,x)_X^2}
=\abs{\Phi}^2_H+8\frac{\pa{x,\Phi}^2_X}{(x,x)_X^2}
\leq_{(*)} \abs{\Phi}^2_H+8\frac{\norm{\Phi}^2_X}{(x,x)_X}\leq \abs{\Phi}^2_H+8K\frac{\abs{\Phi}^2_H}{(x,x)_X},
\end{gather*}
where $(*)$ follows by the Schwarz inequality and $\norm{h}_X\leq K\abs{h}_H$ for every $h\in H$. We can now integrate both terms of the inequality 
\begin{gather*}
\int_X\sum_{i=1}^{+\infty}\sum_{j=1}^{+\infty}(\delta_{ij}-\partial_i\partial_j\log w(x))\varphi_i\varphi_jd\nu\leq \int_X\pa{\abs{\Phi}^2_H+8K\frac{\abs{\Phi}^2_H}{(x,x)_X}}d\nu= \\
=\int_X\abs{\Phi}^2_H(x,x)_X^2d\mu+8K\int_X\abs{\Phi}^2_H(x,x)_Xd\mu.
\end{gather*}
By the H\"older inequality and Fernique's theorem (see \cite[Theorem 2.8.5]{Bog98}) there exists a constant $\ol{K}$ such that
\begin{gather*}
\int_X\sum_{i=1}^{+\infty}\sum_{j=1}^{+\infty}(\delta_{ij}-\partial_i\partial_j\log w(x))\varphi_i\varphi_jd\nu\leq \ol{K}\norm{\Phi}_{\elle^{2\alpha}(X,\mu;H)}^2
\end{gather*}  
Fix now a field of the following form: $\Phi(x)=\sum_{i=1}^n \varphi_i(x)e_i$ with $\varphi_i\in\fcon^{\infty}_b(X)$ for every $i=1,\ldots,n$. Following the same steps of the proof of Proposition \ref{divergence for W12} we rewrite (\ref{formula divergenza}) as
\begin{gather*}
\int_X\pa{\diver_\nu \Phi}^2d\nu\leq \ol{K}\norm{\Phi}^2_{L^{2\alpha}(X,\mu;H)}+\int_X\norm{\nabla_H \Phi(x)}_{\mathcal{H}}^2d\nu(x)\leq\max\{\ol{K},1\}\norm{\Phi}^2_{W^{1,2\alpha}(X,\mu;H)}.
\end{gather*}
By the same argument used at the end of the proof of Proposition \ref{divergence for W12} and by (\ref{W(mu) in W(nu)}) of Proposition \ref{proprerties of sobolev space} (since it implies that $W^{1,2\alpha}(X,\mu;H)\subseteq\elle^1(X,\nu;X)$ for every $\alpha>1$), we obtain that $W^{1,2\alpha}(X,\mu;H)\subseteq D(\diver_\nu)$ for every $\alpha>1$.
\end{eg}

Since in several pratical examples the vector field we are interested in belongs to $W^{1,q}(X,\mu; H)$ for some $q>1$, the following proposition might be useful.

\begin{pro}\label{divergence with mu}
Let $\Phi\in W^{1,q}(X,\mu;H)$ for some $q\geq \frac{s't}{t-s'}$. Then $\diver_\nu \Phi$ exists and it belongs to $\elle^{p}(X,\nu)$ for every $p\in\left[1,\frac{qt}{s'(q+t)}\right)$. Furthermore
\begin{gather}\label{Formula divergenza}
\diver_\nu \Phi=\diver_\mu \Phi+\gen{\Phi,\nabla_H\log w}_H
\end{gather}
and $\norm{\diver_\nu \Phi}_{\elle^{p}(X,\nu)}\leq C_p\norm{\Phi}_{W^{1,q}(X,\mu;H)}$ for some $C_p>0$. Finally if $\Phi\in W^{1,q}(X,\mu;H)$ for every $q>1$, then $\diver_\nu \Phi\in\elle^{r}(X,\nu)$ for every $r\in[1,t/s')$.
\end{pro} 

\begin{proof}
By the H\" older inequality $\gen{\Phi,\nabla_H\log w}_H\in\elle^{\frac{qt}{q+t}}(X,\mu)$ and by \cite[Proposition 5.8.8]{Bog98} $\diver_\mu \Phi\in\elle^q(X,\mu)$. By (\ref{W(mu) in W(nu)}) of Proposition \ref{proprerties of sobolev space} both $\diver_\mu \Phi$ and $\gen{\Phi,\nabla_H\log w}_H$ belong to $\elle^{p}(X,\nu)$ for every $p\in\sq{1,\frac{qt}{s'(q+t)}}$. Let $f\in\fcon_b^\infty(X)$ and $(w_n)_{n\in\N}\subseteq\fcon_b^\infty(X)$ be a sequence of positive functions which converges to $w$ pointwise and in $W^{1,s}(X,\mu)$ (so that also $\log w_n$ converges to $\log w$ both pointwise and in $W^{1,t}(X,\mu)$), then
\begin{gather*}
\int_Xf(\diver_\mu \Phi+\gen{\Phi,\nabla_H\log w_n}_H)w_n d\mu=\int_X(fw_n)\diver_\mu \Phi d\mu+\int_Xf\gen{\Phi,\nabla_H\log w_n}_H w_nd\mu=\\
=-\int_X \gen{\nabla_H(fw_n),\Phi}_Hd\mu+\int_Xf\gen{\Phi,\nabla_H\log w_n}_H w_nd\mu=\\
=-\int_X \gen{\nabla_H f,\Phi}_Hw_nd\mu-\int_X f\gen{\nabla_H(w_n),\Phi}_Hd\mu+\int_Xf\gen{\Phi,\nabla_H\log w_n}_H w_nd\mu=\\
=-\int_X \gen{\nabla_H f,\Phi}_Hw_nd\mu-\int_X f\gen{\nabla_H(\log w_n),\Phi}_Hw_n d\mu+\int_Xf\gen{\Phi,\nabla_H\log w_n}_H w_nd\mu=\\
=-\int_X \gen{\nabla_H f,\Phi}_Hw_nd\mu.
\end{gather*}
Letting $n\ra+\infty$ we obtain that $\diver_\nu \Phi=\diver_\mu \Phi+\gen{\Phi,\nabla_H\log w}_H$. The inequality
\[\norm{\diver_\nu \Phi}_{\elle^{p}(X,\nu)}\leq C_p\norm{\Phi}_{W^{1,q}(X,\mu;H)}\]
follows by a standard H\" older inequality application and \cite[Proposition 5.8.8]{Bog98}.
\end{proof}

\section{Sobolev spaces on sublevel sets}\label{Sobolev spaces on admissible sets}

Let $G$ be a function satisfying Hypothesis \ref{ipotesi2}. We are interested in Sobolev spaces on sets of the form $G^{-1}(-\infty,0)$. Via a standard argument it is possible to prove that $\lip(G^{-1}(-\infty,0))$ is a dense subspace in $\elle^p(G^{-1}(-\infty,0),\nu)$ (see \cite{AT04}). Whenever $p\geq\frac{t}{t-s'}$ it is possible to define
\[\nabla_H^0:\lip(G^{-1}(-\infty,0))\longra\elle^p(G^{-1}(-\infty,0),\nu;H)\]
in the following way: for every $\varphi\in \lip(G^{-1}(-\infty,0))$, with Lipschitz constant $L_\varphi$, consider the McShane extension (see \cite{McS34})
\[\varphi_M(x)=\sup\set{\varphi(y)+L_\varphi\norm{x-y}_X\tc G(y)<0}.\]
It is well known that $\varphi_M\in\lip(X)\subseteq W^{1,q}(X,\nu)$ for some $q\geq\frac{t}{t-s'}$ ((\ref{W(mu) in W(nu)}) of Proposition \ref{proprerties of sobolev space} and \cite[Theorem 5.11.2]{Bog98}). Let
\[\nabla_H^0\varphi:=\nabla_H\varphi_M.\]

\begin{pro}
Let $p\geq \frac{t}{t-s'}$. The operator 
\[\nabla^0_H:\lip(G^{-1}(-\infty,0))\ra L^p(G^{-1}(-\infty,0),\nu; H)\] 
is closable in $\elle^p(G^{-1}(-\infty,0),\nu)$.
\end{pro}

\begin{proof}
Let $(f_k)_{k\in\N}\subseteq \lip(G^{-1}(-\infty,0))$ such that
\begin{align*}
&\lim_{k\ra+\infty}f_k=0\qquad\text{ in }L^p(G^{-1}(-\infty,0),\nu);\\
&\lim_{k\ra+\infty}\nabla_H^0 f_k=\Phi\qquad\text{ in }L^p(G^{-1}(-\infty,0),\nu;H).
\end{align*}
We want to prove that
\[\int_{G^{-1}(-\infty,0)}\gen{\Phi(x),e_i}u(x)d\nu(x)=0\]
for every $i\in\N$ and $u\in\fcon^{\infty}_b(X)$. Let $\eta:\R\ra\R$ a smooth function such that $\norm{\eta}_\infty\leq 1$, $\norm{\eta'}_\infty\leq 2$ and
\[\eta(\xi)=\eqsys{0 & \xi\geq -1\\
1 & \xi\leq -2}\]
Let $\eta_n(\xi):=\eta(n\xi)$ and $u_n(x)=u(x)\eta_n(G(x))$. Observe that $u_n$ converges pointwise $\nu$-a.e. to $u$ on $G^{-1}(-\infty,0)$ and $\abs{u_n}\leq \abs{u}$ $\nu$-a.e., then by Lebesgue's dominated convergence theorem
\[\lim_{n\ra+\infty}\int_{G^{-1}(-\infty,0)}\gen{\Phi(x),e_i}u_n(x)d\nu(x)=\int_{G^{-1}(-\infty,0)}\gen{\Phi(x),e_i}u(x)d\nu(x).\]
By (\ref{prodotto funzioni in W}) of Proposition \ref{proprerties of sobolev space}, for every $n\in\N$ we have $u_n\in W^{1,r}(X,\nu)$, for every $r\geq\frac{t}{t-s'}$, and by Proposition \ref{C1 func}
\[\partial_i u_n(x)=\partial_i u(x)\eta_n(G(x))+u(x)\eta_n'(G(x))\partial_i G(x).\]
Observe that
\begin{gather*}
\int_X u_n\partial_i f_kd\nu=\int_X f_ku_n(\hat{e}_i-\partial_i\log w)d\nu-\int_X f_k \partial_i u(\eta_n\circ G)d\nu-\int_X f_k u(\eta'\circ G)\partial_i Gd\nu,
\end{gather*}
and the following estimates holds:
\begin{gather*}
\int_X\abs{\partial_i f_k u_n-\gen{\Phi,e_i}_H u}d\nu\leq \int_X\abs{\partial_i f_k}\abs{u_n-u}d\nu+\int_X\abs{\partial_i f_k-\gen{\Phi,e_i}_H}\abs{u}d\nu\leq\\
\leq\pa{\int_X\abs{\partial_i f_k}^pd\nu}^{\frac{1}{p}}\pa{\int_X\abs{u_n-u}^{p'}d\nu}^{\frac{1}{p'}}+\pa{\int_X\abs{\partial_i f_k-\gen{\Phi,e_i}_H}^p d\nu}^{\frac{1}{p}}\pa{\int_X\abs{u}^{p'}d\nu}^{\frac{1}{p'}},
\end{gather*}
this means $\lim_{n\ra+\infty}\lim_{k\ra+\infty}\int_X u_n\partial_i f_kd\nu=\int_{G^{-1}(-\infty,0)}\gen{\Phi,e_i}u d\nu$. Furthermore for every $n\in\N$ we get
\begin{gather*}
\int_X\abs{f_k\partial_i u(\eta_n\circ G)} d\nu\leq \int_X\abs{f_k\partial_i u}d\nu\leq\pa{\int_X\abs{f_k}^pd\nu}^{\frac{1}{p}}\pa{\int_X\abs{\partial_i u}^{p'}d\nu}^{\frac{1}{p'}}\xra{k\ra+\infty}0;\\
\int_X\abs{f_k u(\eta'_n\circ G)\partial_i G}d\nu\leq\norm{u}_\infty\pa{\int_X\abs{f_k}^p d\nu}^{\frac{1}{p}}\pa{\int_X\abs{\partial_i G}^{p's'} d\mu}^{\frac{1}{p's'}}\pa{\int_X w^sd\mu}^{\frac{1}{p's}}\xra{k\ra+\infty}0,\\
\intertext{where the last limit follows by to Hypothesis \ref{ipotesi2};}
\int_X\abs{f_k u_n\hat{e}_i}d\nu\leq\norm{u}_\infty\pa{\int_X\abs{f_k}^pd\nu}^{\frac{1}{p}}\pa{\int_X\abs{\hat{e}_i}^{p'}d\nu}^{\frac{1}{p'}}\xra{k\ra+\infty}0;\\
\int_X\abs{f_k u_n \partial_i \log w}d\nu\leq\norm{u}_\infty\pa{\int_X\abs{f_k}^pd\nu}^{\frac{1}{p}}\pa{\int_X\abs{w}^sd\mu}^{\frac{1}{p's}}\pa{\int_X\abs{\partial_i \log w}^{p's'}d\mu}^{\frac{1}{p's'}}\xra{k\ra+\infty}0,
\end{gather*}
and the last limit exists whenever $p's'\leq t$.
\end{proof}

\begin{defn}[Weighted Sobolev space on sublevel sets]
Let $p\geq\frac{t}{t-s'}$. We denote by $W^{1,p}(G^{-1}(-\infty,0),\nu)$ the domain of the closure of the operator $\nabla_H^0$ (which we will denote by the symbol $\nabla_H$) in $\elle^p(G^{-1}(-\infty,0),\nu)$. It is a Banach space with the graph norm
\[\norm{f}_{W^{1,p}(G^{-1}(-\infty,0),\nu)}=\pa{\int_{G^{-1}(-\infty,0)}\abs{f(x)}^pd\nu(x)}^{\frac{1}{p}}+\pa{\int_{G^{-1}(-\infty,0)}\abs{\nabla_H f(x)}_H^pd\nu(x)}^{\frac{1}{p}}.\]
\end{defn}

The following equalities are what will allow us to define the trace operator in the following section.

\begin{pro}\label{Divergence without trace}
Let $k\in \N$. Then for every $\varphi\in\lip_b(G^{-1}(-\infty,0))$ and $G$ satisfying Hypothesis \ref{ipotesi2}, we have
\[\int_{G^{-1}(-\infty,0)}\pa{\partial_k\varphi+\varphi\partial_k\log w-\varphi\hat{e}_k}d\nu=\int_{G^{-1}(0)}\pa{\frac{\varphi \partial_k G}{\abs{\nabla_H G}_H}}_{_{|_{G^{-1}(0)}}}wd\rho.\]
\end{pro}

\begin{proof}
The proof is the same as \cite[Equation (1.1), proof in Proposition 4.1]{CL14}.
\end{proof}

\begin{pro}\label{traces estimates}
Let $q\geq 1$. Then for every $\varphi\in\lip_b(G^{-1}(-\infty,0))$ and $G$ satisfying Hypothesis \ref{ipotesi2}, we have
\[\int_{G^{-1}(-\infty,0)}\pa{q\varphi\abs{\varphi}^{q-2}\gen{\nabla_H\varphi,\nabla_H G}_H-\abs{\varphi}^q \diver_\nu \nabla_H G}d\nu=\int_{G^{-1}(0)}(\abs{\varphi}^q\abs{\nabla_H G}_H)_{_{|_{G^{-1}(0)}}}wd\rho,\]
and
\[\int_{G^{-1}(-\infty,0)}\pa{q\varphi\abs{\varphi}^{q-2}\gen{\nabla_H\varphi,\frac{\nabla_H G}{\abs{\nabla_H G}_H}}_H+\diver_\nu\frac{\nabla_H G}{\abs{\nabla_H G}_H}\abs{\varphi}^q}d\nu=\int_{G^{-1}(0)}(\abs{\varphi}^q)_{_{|_{G^{-1}(0)}}}wd\rho.\]
\end{pro}

\begin{proof}
The proof is the same as \cite[Proposition 4.1]{CL14}.
\end{proof}

\section{Traces of Sobolev functions on sublevel sets}\label{Traces}

Throughout this section we will denote by $G$ a function satisfying Hypothesis \ref{ipotesi2}. The following result is fundamental for the definition of the trace operator.

\begin{pro}\label{constistenza definizione traccia}
Let $p\geq \frac{t}{t-s'}$. The following holds:
\begin{enumerate}
\item if $(\varphi_n)_{n\in\N}\subseteq\lip_b(G^{-1}(-\infty,0))$ is a Cauchy sequence in $W^{1,p}(G^{-1}(-\infty,0),\nu)$, then $(\varphi_{n_{|_{G^{-1}(0)}}})$ is a Cauchy sequence in $\elle^q(G^{-1}(0),w\rho)$ for every $1\leq q\leq p\frac{t-s'}{t}$;

\item if $(\varphi_n)_{n\in\N},(\psi_n)_{n\in\N}\subseteq\lip_b(G^{-1}(-\infty,0))$ converge to $\varphi$ in $W^{1,p}(G^{-1}(-\infty,0),\nu)$, then $(\varphi_{n_{|_{G^{-1}(0)}}})$ and $(\psi_{n_{|_{G^{-1}(0)}}})$ converge to the same element in $\elle^q(G^{-1}(0),w\rho)$ for every $1\leq q\leq p\frac{t-s'}{t}$.
\end{enumerate}
\end{pro}

\begin{proof}
Let $(\varphi_n)_{n\in\N}\subseteq\lip_b(G^{-1}(-\infty,0))$ be a Cauchy sequence in $W^{1,p}(G^{-1}(-\infty,0),\nu)$. By Proposition \ref{traces estimates} we have for every $q\in\left[1,p\frac{t-s'}{t}\right]$
\begin{gather*}
\int_{G^{-1}(0)}\abs{\varphi_n-\varphi_m}^qwd\rho
=\int_{G^{-1}(-\infty,0)}\pa{q\abs{\varphi_n-\varphi_m}^{q-2}(\varphi_n-\varphi_m)\gen{\nabla_H(\varphi_n-\varphi_m),\frac{\nabla_H G}{\abs{\nabla_H G}_H}}_H}d\nu+\\
+\int_{G^{-1}(-\infty,0)}\pa{\diver_\nu\frac{\nabla_H G}{\abs{\nabla_H G}_H}\abs{\varphi_n-\varphi_m}^q}d\nu\xra{n,m\ra+\infty}0.
\end{gather*}
All the convergences are assured by Proposition \ref{divergence with mu}. Thus (1) is proved.

Let $(\varphi_n)_{n\in\N},(\psi_n)_{n\in\N}\subseteq\lip_b(G^{-1}(-\infty,0))$ satisfying (2) and let $\varphi_\infty$ the limit of $(\varphi_{n_{|_{G^{-1}(0)}}})$ in $\elle^q(G^{-1}(0),w\rho)$ for every $1\leq q\leq p\frac{t-s'}{t}$ (it exists by (1)). By Proposition \ref{traces estimates} we have for every $q\in\left[1,p\frac{t-s'}{t}\right]$
\begin{gather*}
\int_{G^{-1}(0)}\abs{\psi_n-\varphi_\infty}^qwd\rho\leq 2^{q-1}\pa{\int_{G^{-1}(0)}\abs{\psi_n-\varphi_n}^qwd\rho+\int_{G^{-1}(0)}\abs{\varphi_n-\varphi_\infty}^qwd\rho}=\\
=2^{q-1}\biggl(\int_{G^{-1}(-\infty,0)}\pa{q\abs{\psi_n-\varphi_n}^{q-2}(\psi_n-\varphi_n)\gen{\nabla_H(\psi_n-\varphi_n),\frac{\nabla_H G}{\abs{\nabla_H G}_H}}_H}d\nu+\\
+\int_{G^{-1}(-\infty,0)}\pa{\diver_\nu\frac{\nabla_H G}{\abs{\nabla_H G}_H}\abs{\psi_n-\varphi_n}^q}d\nu+\int_{G^{-1}(0)}\abs{\varphi_n-\varphi_\infty}^qwd\rho\biggr )\xra{n,m\ra+\infty}0.
\end{gather*}
Thus (2) is proved.
\end{proof}

We are now able to define the trace of a Sobolev function.

\begin{defn}\label{Trace definition}
Let $p\geq \frac{t}{t-s'}$. If $\varphi\in W^{1,p}(G^{-1}(-\infty,0),\nu)$ we define the trace of $\varphi$ on $G^{-1}(0)$ as follows:
\[\trace_{G^{-1}(0)}\varphi=\lim_{n\ra+\infty}\varphi_{n_{|_{G^{-1}(0)}}}\qquad\text{in }\elle^{q}(G^{-1}(0),w\rho)\text{ for every } q\in\left[1,p\frac{t-s'}{t}\right],\]
where $(\varphi_{n})_{n\in\N}$ is any sequence in $\lip_b(G^{-1}(-\infty,0))$ which converges in $W^{1,p}(G^{-1}(-\infty,0),\nu)$ to $\varphi$. By Proposition \ref{constistenza definizione traccia} the definition does not depend on the choice of the sequence $(\varphi_n)_{n\in\N}$ in $\lip_b(G^{-1}(-\infty,0))$ approximating $\varphi$ in $W^{1,p}(G^{-1}(-\infty,0),\nu)$.
\end{defn}

Observe that by \cite[Proposition 4.8]{CL14} $w\rho=w_{|_{G^{-1}(0)}}\rho$. An obvious consequence of the definition is the following Corollary.

\begin{cor}\label{Trace continuity}
Let $p\geq \frac{t}{t-s'}$. The operator
\[\trace_{G^{-1}(0)}:W^{1,p}(G^{-1}(-\infty,0),\nu)\longra\elle^q(X,w\rho)\]
is continuous for every $q\in\left[1,p\frac{t-s'}{t}\right]$.
\end{cor}

\begin{pro}\label{Trace product}
Let $a>1$ and $b>1$. If $\frac{ab}{a+b}\geq\frac{t}{t-s'}$, then for every $\varphi\in W^{1,a}(X,\nu)$ and $\psi\in W^{1,b}(X,\nu)$ we have
\[\trace_{G^{-1}(0)}(\varphi\psi)=\trace_{G^{-1}(0)}(\varphi)\trace_{G^{-1}(0)}(\psi)\qquad \rho\text{-a.e.}\]
\end{pro}

\begin{proof}
Let $(\varphi_n)_{n\in\N},(\psi_n)_{n\in\N}\subseteq \lip_b(G^{-1}(-\infty,0))$ such that
\begin{gather*}
\lim_{n\ra+\infty}\varphi_n=\varphi\qquad\text{ in }W^{1,a}(G^{-1}(-\infty,0),\nu);\\
\lim_{n\ra+\infty}\psi_n=\psi\qquad\text{ in }W^{1,b}(G^{-1}(-\infty,0),\nu).
\end{gather*}
By a standard argument we know $\varphi_n\psi_n\in\lip_b(G^{-1}(-\infty,0))$, for every $n\in\N$, and
\[\lim_{n\ra+\infty}\varphi_n\psi_n=\varphi\psi\qquad\text{ in }W^{1,\frac{ab}{a+b}}(G^{-1}(-\infty,0),\nu).\]
So we have
\[\trace_{G^{-1}(0)}(\varphi\psi)=\lim_{n\ra+\infty}\varphi_{n_{|_{G^{-1}(0)}}}\psi_{n_{|_{G^{-1}(0)}}}\quad\text{in }\elle^{q}(G^{-1}(0),w\rho)\text{ for every } q\in\left[1,\frac{ab}{a+b}\frac{t-s'}{t}\right].\]
Using H\"older inequality we get
\begin{gather*}
\int_X|\varphi_{n_{|_{G^{-1}(0)}}}\psi_{n_{|_{G^{-1}(0)}}}-\trace_{G^{-1}(0)}(\varphi)\trace_{G^{-1}(0)}(\psi)|wd\rho\xra{n\ra+\infty}0.
\end{gather*}
So $\trace_{G^{-1}(0)}(\varphi\psi)=\trace_{G^{-1}(0)}(\varphi)\trace_{G^{-1}(0)}(\psi)$ $\rho$-a.e.
\end{proof}

\begin{pro}\label{trace and precise version}
Let $p\geq\frac{t}{t-s'}$. For every $\varphi\in W^{1,p}(G^{-1}(-\infty,0),\nu)$ we have
\[\trace_{G^{-1}(0)}\varphi(x)=\ol{\varphi}_{|_{G^{-1}(0)}}(x)\qquad\rho\text{-a.e.}\]
for every $(1,p)$-precise version $\ol{\varphi}$ of $\varphi$.
\end{pro}

\begin{proof}
Let $(\varphi_n)_{n\in\N}$ a sequence of bounded Lipschitz functions defined on $G^{-1}(-\infty,0)$ which satisties the condition of the Definition \ref{Trace definition}. By Proposition \ref{traces estimates} we get
\begin{gather*}
\int_{G^{-1}(0)}\abs{\ol{\varphi}_{|_{G^{-1}(0)}}-\varphi_n}\abs{\nabla_H G}_Hwd\rho=\\
=\int_{G^{-1}(-\infty,0)}\pa{\text{sign}(\ol{\varphi}-\varphi_n)\gen{\nabla_H(\ol{\varphi}-\varphi_n),\nabla_H G}_H-\abs{\ol{\varphi}-\varphi_n}\diver_\nu\nabla_H G}d\nu.
\end{gather*}
Letting $n\ra+\infty$ we get $\trace_{G^{-1}(0)}\varphi=\ol{\varphi}_{|_{G^{-1}(0)}}$ $\rho$-a.e.
\end{proof}

We are now in a position to prove Theorem \ref{divergence theorem with traces}.

\begin{proof}[Proof of Theorem \ref{divergence theorem with traces}]
The statement follows by Proposition \ref{Divergence without trace} and Proposition \ref{trace and precise version}. The furthermore part follows by Proposition \ref{Trace product}.
\end{proof}

\noindent Using the same argument we can extend  Proposition \ref{traces estimates} to functions in $W^{1,p}(G^{-1}(-\infty,0),\nu)$.

\begin{pro}\label{traces estimates for W1p}
Let $p\geq \frac{t}{t-s'}$ and $1\leq q\leq p\frac{t-s'}{t}$. If $\varphi\in W^{1,p}(G^{-1}(-\infty,0),\nu)$ then
\begin{gather*}
\int_{G^{-1}(-\infty,0)}\pa{q\varphi\abs{\varphi}^{q-2}\gen{\nabla_H\varphi,\nabla_H G}_H-\abs{\varphi}^q \diver_\nu \nabla_H G}d\nu=\\
=\int_{G^{-1}(0)}\trace_{G^{-1}(0)}\abs{\varphi}^q\trace_{G^{-1}(0)}\abs{\nabla_H G}_Hwd\rho,
\end{gather*}
and
\[\int_{G^{-1}(-\infty,0)}\pa{q\varphi\abs{\varphi}^{q-2}\gen{\nabla_H\varphi,\frac{\nabla_H G}{\abs{\nabla_H G}_H}}+\diver_\nu\frac{\nabla_H G}{\abs{\nabla_H G}_H}\abs{\varphi}^q}d\nu=\int_{G^{-1}(0)}(\trace_{G^{-1}(0)}\abs{\varphi}^q)wd\rho.\]
\end{pro}

\section{Examples}\label{Examples}

In this section we show how our results may be applied to some explicit examples. Recall that by $\partial_i f$ we denote the partial derivative of $f$ along the direction $e_i\in H$ (see Section \ref{Notations and preliminaries}). We will be interested in two types of surfaces:
\begin{description}
\item[Unit sphere] Let $S(x)=\norm{x}_X-1$. We will prove that $S$ satisfies Hypothesis \ref{ipotesi2} in our examples.

\item[Hyperplanes] Let $f\in X^*\ssm\set{0}$ and $G_f(x)=f(x)$. Observe that 
$\partial_i G_f(x)=f(e_i)$ for every $i\in\N$ and $\partial_i\partial_j G_f(x)=0$ for every $i,j\in\N$. Furthermore
\[G_f\in\bigcap_{p>1} W^{2,p}(X,\mu)\quad\text{ and }\quad\frac{1}{\abs{\nabla_H G_f}_H}\in\bigcap_{p>1}\elle^p(X,\mu).\]
Thus $G_f$ satifies Hypothesis \ref{ipotesi2}, for every $f\in X^*\ssm\set{0}$.
\end{description}

\subsection{A Gaussian-type weight in Hilbert spaces}\label{esempio spazio hilbert}

Let $X$ be a separable Hilbert space endowed with a nondegenerate centered Gaussian measure $\mu$, with covariance $Q$. Fix an orthonormal basis $(v_n)_{n\in\N}$ of $X$ of eigenvectors of $Q$, i.e. $Qv_k=\lambda_k v_k$, and the corresponding orthonormal basis of $H=Q^{\frac{1}{2}}(X)$ is $\{e_n:=\sqrt{\lambda_n}v_n\}_{n\in\N}$. 

Let $w_\lambda(x)=e^{\lambda(x,x)_X}$ for $\lambda\in\R$. Easy calculation gives
\[\partial_i w_\lambda(x)=2\lambda(x,e_i)_Xe^{\lambda(x,x)_X},\quad\partial_i\log w_\lambda(x)=2\lambda(x,e_i)_X,\quad\partial_j\partial_i\log w_\lambda(x)=2\lambda(e_j,e_i)_X.\]
Let 
\[\alpha:=\sup\set{\eta>0\tc\int_X e^{\eta(x,x)_X}d\mu(x)<+\infty}.\]
By Fernique's theorem (see \cite[Theorem 2.8.5]{Bog98}) the set $\set{\eta>0\tc\int_X e^{\eta(x,x)_X}d\mu(x)<+\infty}$ is not empty and $\alpha$ is strictly positive.  Furthermore
\begin{gather*}
\int_X e^{\eta(x,x)_X}d\mu(x)=\int_X e^{\eta\sum_{i=1}^{+\infty}(x,v_i)_X^2}d\mu(x)=\lim_{n\ra+\infty}\int_X e^{\eta\sum_{i=1}^{n}(x,v_i)_X^2}d\mu(x)=\\
=\lim_{n\ra+\infty}\prod_{i=1}^n\sqrt{\frac{\lambda_i}{2\pi}}\int_{\R}e^{(\eta-\frac{\lambda_i}{2})\xi^2}d\xi,
\end{gather*}
and the last limit diverges if $\eta\geq 2\lambda_1$. Thus $0<\alpha<2\lambda_1$ and $w_\lambda\in W^{1,s}(X,\mu)$ for every $2\lambda<\alpha$ whenever $\lambda(s+1)\leq \alpha$. Furthermore $\log w_\lambda\in W^{2,t}(X,\mu)$ for every $t>1$ and $2\lambda<\alpha$.
In both cases it is possible to define $W^{1,p}(X,\nu)$ for every $p>1$ (see Definition \ref{Weightes Sobolev space definition}).

The above observation gives also that $S$ satisfies Hypothesis \ref{ipotesi2}. In this case $w_\lambda\rho=e^\lambda\rho$, so all the remarks about continuity of the trace operator in $\elle^p(S^{-1}(0),\rho)$ stated in \cite[Section 5.3]{CL14} still hold. 

By Corollary \ref{Trace continuity}, the trace operator $\trace_{G_f^{-1}(0)}:W^{1,p}(G_f^{-1}(-\infty,0),\nu_\lambda)\ra\elle^q(G_f^{-1}(0),w_\lambda\rho)$ for every $q\in [1,p)$ and $f\in X^*\ssm\set{0}$. Furthermore, using a similar argument as in Proposition \ref{proprerties of sobolev space}, we get $\trace_{G_f^{-1}(0)}\varphi\in\elle^q(G_f^{-1}(0),\rho)$ for every $q\in[1,p)$ and every $\varphi\in W^{1,p}(G^{-1}_f(-\infty,0),\nu)$. 

\subsection{A weight without continuous versions}\label{esempio ell2}

Let $X=\ell_2$ the Banach space of square summable sequences and let $(v_k)_{k\in\N}$ its standard orthonormal basis, i.e. $v_k$ is the sequence such that $v_k(i)=\delta_{ik}$ for every $i,k\in\N$. Let $\mu$ be a centered non-degenerate Gaussian measure on $\ell_2$ with covariance operator $Q:\ell_2\ra\ell_2$ defined by
\[Q(x)=\pa{\frac{x(i)}{2^i}}.\]
Such a measure exists, e.g. by \cite[Theorem 2.3.1]{Bog98}. The eigenvectors of $Q$ are the vectors $v_k$ with respective eigenvalues $2^{-k}$.
We will denote by $\{e_n:=v_n/\sqrt{2^n}\}$ the basis of the Cameron--Martin space associated with $\mu$.

The weight we want to study is $w_q(x)=e^{\norm{x}_q}$ for fixed $q> 1$. The first result we need is the fact that $w_q$ (actually $\norm{\cdot}_q$) is defined $\mu$-a.e. and in order to show that we need a modification of Fernique's theorem (see \cite[Theorem 2.8.5]{Bog98}). Let's start with a definition:

\begin{defn}
Let $\gamma$ be a Gaussian measure on a separable Banach space $X$ and $p\in(0,1]$. A function $g$ measurable with respect to $\borel(X)$ (the Borel $\sigma$-algebra of $X$) is called \emph{$\borel(X)$-measurable $p$-seminorm} if there exists a $\borel(X)$-measurable linear subspace $X_0\subseteq X$ of $\gamma$-measure $1$ such that $g$ is a $p$-seminorm on $X_0$, i.e. $g(x+y)\leq g(x)+g(y)$ for every $x,y\in X_0$ and $g(\lambda x)=\abs{\lambda}^pg(x)$ for every $x\in X_0$ and $\lambda\in\R$.
\end{defn}
Our definition differs from the definition of measurable norm of the Cameron--Martin
space with the cylindrical Wiener measure in the sense of Gross (see \cite[Definition 3.9.2]{Bog98}) since we also consider $p$-seminorm, with $0<p<1$. Observe that our definition of $\borel(X)$-measurable 1-seminorm agrees with \cite[Definition 2.8.1]{Bog98}.

The proofs of the following two propositions are similar to the proof of Fernique's theorem. We include the proof of the first one in order to get a self-contained paper.

\begin{pro}\label{Fernique with p-norm}
Let $\mu$ be a centered Gaussian measure on a separable Banach space $X$, $p\in (0,1]$ and let $g$ be a $\borel(X)$-measurable $p$-seminorm. Then
\[\int_Xe^{\alpha g(x)^2}d\mu(x)<+\infty,\]
for some $\alpha>0$.
\end{pro}

\begin{proof}
Let $t>\tau>0$. According to \cite[Proposition 2.2.10]{Bog98} we have
\begin{gather*}
\mu\pa{\set{x\in X\tc g(x)\leq\tau}}\mu\pa{\set{y\in X\tc g(y)>t}}=
\int_{\set{(x,y)\in X\times X\tc g(x)\leq \tau\text{ and }g(y)>t}}d\mu(x)d\mu(y)=\\
=\int_{\set{(u,v)\in X\times X\tc g\pa{\frac{u-v}{\sqrt{2}}}\leq \tau\text{, }g\pa{\frac{u+v}{\sqrt{2}}}>t}}d\mu(u)d\mu(v)\leq\\
\leq \int_{\set{(u,v)\in X\times X\tc g\pa{u}\geq \frac{t-\tau}{\sqrt{2^p}}\text{, }g\pa{v}>\frac{t-\tau}{\sqrt{2^p}}}}d\mu(u)d\mu(v).
\end{gather*}
The last inequality follows by $g(u)\geq 2^{-p}(g(u+v)-g(u-v))$ for every $u,v\in X$, indeed
\[\set{(u,v)\in X^2\tc g\pa{\frac{u-v}{\sqrt{2}}}\leq \tau\text{, }g\pa{\frac{u+v}{\sqrt{2}}}>t}\subseteq \set{(u,v)\in X^2\tc g\pa{u}\geq \frac{t-\tau}{\sqrt{2^p}}\text{, }g\pa{v}>\frac{t-\tau}{\sqrt{2^p}}}.\]
Therefore we get
\begin{gather*}
\mu\pa{\set{x\in X\tc g(x)\leq\tau}}\mu\pa{\set{y\in X\tc g(y)>t}}\leq\pa{\mu\set{z\in X\tc g(z)>\frac{t-\tau}{\sqrt{2^p}}}}^2
\end{gather*}
Since $g<+\infty$ $\mu$-a.e., there exists a positive number $\tau$ such that $c:=\mu\pa{\set{x\in X\tc g(x)\leq \tau}}>\frac{1}{2}$. If $c=1$ the statement holds true, indeed $\int_Xe^{\alpha g(x)^2}d\mu(x)\leq e^{\alpha\tau^2}$ for every $\alpha>0$.
Now assume $c<1$. Let
\[t_n=\tau+\sqrt{2^p}t_{n-1}\quad\text{ and }\quad t_0=\tau.\]
It is easy to verify that $t_n=\tau(\sqrt{2^p}-1)^{-1}(2^{p\frac{n+1}{2}}-1)$.
Letting $p_n:=c^{-1}\mu\pa{\set{x\in X\tc g(x)>t_n}}$, then $p_n\leq p_{n-1}^2$. By induction we get
\[\mu\pa{\set{x\in X\tc g(x)>t_n}}\leq c\pa{\frac{1-c}{c}}^{2^n}.\]
Let $\alpha> 0$. We get
\begin{gather*}
\int_X e^{\alpha g(x)^2}d\mu(x)\leq \int_{\set{x\in X\tc g(x)\leq \tau}}e^{\alpha g(x)^2}d\mu(x)+\sum_{n=0}^{+\infty}e^{\alpha t_{n+1}^2}\mu\pa{\set{x\in X\tc t_n\leq g(x)< t_{n+1}}}\leq\\
\leq ce^{\alpha \tau^2}+\sum_{n=0}^{+\infty}e^{\alpha t_{n+1}^2}\mu\pa{\set{x\in X\tc g(x)\geq t_n}}\leq ce^{\alpha \tau^2}+\sum_{n=0}^{+\infty}c\pa{\frac{1-c}{c}}^{2^n}\text{exp}\pa{\alpha\tau^2\frac{(2^{p\frac{n+1}{2}}-1)^2}{(\sqrt{2^p}-1)^2}}=\\
=ce^{\alpha \tau^2}+c\sum_{n=0}^{+\infty}\text{exp}\pa{2^n\log\pa{\frac{1-c}{c}}+\alpha\tau^2\pa{\frac{2^{p\frac{n+1}{2}}-1}{\sqrt{2^p}-1}}^2}
\end{gather*}
Observe that by the chain of inequalities $(2^{p\frac{n+1}{2}}-1)^2\leq (2^{\frac{n+1}{2}}-1)^2\leq 2^{n+1}+1$, we get 
\[2^n\log\pa{\frac{1-c}{c}}+\alpha\tau^2\pa{\frac{2^{p\frac{n+1}{2}}-1}{\sqrt{2^p}-1}}^2\leq 2^n\log\pa{\frac{1-c}{c}}+\alpha\tau^2\frac{2^{n+1}+1}{\pa{\sqrt{2^p}-1}^2}.\]
So if we choose $\alpha>0$ such that
\[\log\pa{\frac{1-c}{c}}+\frac{2\alpha\tau^2}{\pa{\sqrt{2^p}-1}^2}<0,\]
then $\int_X e^{\alpha g^2}d\mu$ is finite.
\end{proof}

\begin{pro}\label{Fernique without square}
Let $\mu$ be a centered Gaussian measure on a locally convex space $X$ and let $g$ be a measurable $1$-seminorm. Then 
$\int_Xe^{\alpha g(x)}d\mu(x)<+\infty$,
for every $\alpha>0$.
\end{pro}

\begin{proof}
The proof is similar to the proof of Proposition \ref{Fernique with p-norm}.
\end{proof}

We are now going to consider a notable example.

\begin{pro}\label{norma integrabile}
Let $q>0$ and $\ell_q=\{x\in\ell_2\,|\, \sum_{i=1}^{+\infty}\abs{(x,v_i)_{\ell_2}}^q<+\infty\}$ (observe that $\ell_q=\ell_2$ for every $q\geq 2$).
Then $\mu(\ell_q)=1$ and the function 
\[P_q(x):=\eqsys{\sum_{i=1}^{+\infty}\abs{(x,v_i)_{\ell_2}}^q & q\in(0,1);\\
\norm{x}_q & q\geq 1,}\]
belongs to $\elle^p(\ell_2,\mu)$, for every $p\geq 1$. 
\end{pro}

\begin{proof}
By \cite[Exercise A.3.34]{Bog98}, for every $q>0$ the linear space $\ell_q$ is measurable with respect to the Borel $\sigma$-algebra in $\ell_2$.
We have
\begin{gather*}
\int_{\ell_2}\sum_{i=1}^{n}\abs{(x,v_i)_{\ell_2}}^qd\mu(x)=\sum_{i=1}^{n}\int_{\ell_2}\abs{(x,v_i)_{\ell_2}}^qd\mu(x).
\end{gather*}
Using the change of variable formula (see \cite[Equation (A.3.1)]{Bog98}) we get
\begin{gather*}
\int_{\ell_2}\sum_{i=1}^{n}\abs{(x,v_i)_{\ell_2}}^qd\mu(x)=\sum_{i=1}^{n}\frac{2\sqrt{2^{i-1}}}{\sqrt{\pi}}\int_0^{+\infty}\eta^qe^{-2^{i-1}\eta^2}d\eta.
\end{gather*}
Let $c_q=2\sqrt{\pi^{-1}}\int_0^{+\infty}t^qe^{-t^2}dt$. Then
\begin{gather*}
\int_{\ell_2}\sum_{i=1}^{n}\abs{(x,v_i)_{\ell_2}}^qd\mu(x)=c_q\sum_{i=1}^{n}\pa{\frac{1}{2^{\frac{q}{2}}}}^i= c_q\frac{2^{\frac{q}{2}(n+1)}-1}{2^{\frac{q}{2}n}(2^{\frac{q}{2}}-1)}\leq c_q\frac{2^{\frac{q}{2}}}{2^{\frac{q}{2}}-1}.
\end{gather*}
Then by the monotone convergence theorem we get
\[\int_{\ell_2}\sum_{i=1}^{+\infty}\abs{(x,v_i)_{\ell_2}}^qd\mu(x)<+\infty.\] 
By \cite[Theorem 2.5.5]{Bog98} we have $\mu(\ell_q)=1$.
Observe that $P_q$ is a $\borel(X)$-measurable $q$-norm, for every $q> 0$, then by Proposition \ref{Fernique with p-norm} it belongs to $\elle^p(\ell_2,\mu)$ for every $p\geq 1$.
\end{proof}

Proposition \ref{norma integrabile} implies that $w_q=e^{\norm{x}_q}$ is defined $\mu$-a.e. on $\ell_2$ for every $q>1$. Now we need to prove that $w_q$ satisfies Hypothesis \ref{ipotesi1}. We start with another modification of Fernique's theorem (see \cite[Theorem 2.8.5]{Bog98}), which implies that for every $q>1$, $w_q\in\elle^s(\ell_2,\mu)$ for every $s\geq 1$.

\begin{pro}\label{norma e' sobolev}
For every $q> 1$ consider the function $U_q:\ell_2\ra\R$ defined by
\[U(x)=\eqsys{\norm{x}_q & x\in\ell_q;\\ 0 & x\notin\ell_q.}\]
Then $U\in W^{1,p}(\ell_2,\mu)$ for every $p\geq 1$ and
\[\partial_i U(x)=2^{-\frac{i}{2}}\sign(x,v_i)(x,v_i)^{q-1}\norm{x}_q^{1-q}\qquad\text{ $\mu$-a.e.}\]
\end{pro}

\begin{proof}
For every $n\in\N$ consider the function $\varphi_n:\R^n\ra\R$ defined as
\[\varphi_n(\eta_1,\ldots,\eta_n):=\pa{\sum_{i=1}^n\pa{\eta_i^2+\frac{1}{2^n}}^{\frac{q}{2}}}^{\frac{1}{q}}.\]
Let $U_n(x):=\varphi_n((x,v_1),\ldots,(x,v_n))$ and observe that $U_n\ra U$ pointwise $\mu$-a.e. and $U_n\in\fcon_b^\infty(\ell_2)$ for every $n\in\N$. Indeed for every $n\geq 2$
\begin{gather*}
\pa{\sum_{i=1}^n\abs{(x,v_i)}^q}^{\frac{1}{q}}\leq \abs{U_n(x)}\leq\pa{\max\set{1,2^{q-1}}\sum_{i=1}^n\pa{\abs{(x,v_i)}^q+\frac{1}{2^{\frac{nq}{2}}}}}^{\frac{1}{q}}=\\
=\max\set{1,2^{\frac{q-1}{q}}}\pa{\frac{n}{2^{\frac{nq}{2}}}+\sum_{i=1}^n\abs{(x,v_i)}^q}^{\frac{1}{q}}\leq\max\set{1,2^{\frac{q-1}{q}}}\pa{\frac{n}{2^{\frac{n}{2}}}+\pa{\sum_{i=1}^n\abs{(x,v_i)}^q}^{\frac{1}{q}}}\\
\leq \max\set{1,2^{\frac{q-1}{q}}}\pa{1+\norm{x}_q}.
\end{gather*}
By Lebesgue's dominated convergence theorem and Proposition \ref{norma integrabile} we get $U_n\ra U$ in $\elle^p(\ell_2,\mu)$ for every $p\geq 1$. Observe that
\[\partial_i U_n(x)=\frac{(x,v_i)}{2^{\frac{i}{2}}}\pa{(x,v_i)^2+\frac{1}{2^n}}^{\frac{q}{2}-1}\pa{\sum_{k=1}^n\pa{(x,v_k)^2+\frac{1}{2^n}}^{\frac{q}{2}}}^{\frac{1}{q}-1},\]
and $\partial_i U_n(x)\ra 2^{-\frac{i}{2}}\sign(x,v_i)(x,v_i)^{q-1}\norm{x}_q^{1-q}$ pointwise $\mu$-a.e. 
For $1< q\leq 2$
\begin{gather*}
\abs{\nabla_H U_n(x)}_H=\pa{\sum_{i=1}^{+\infty}\frac{(x,v_i)^2}{2^i}\pa{(x,v_i)^2+\frac{1}{2^n}}^{q-2}\pa{\sum_{k=1}^n\pa{(x,v_k)^2+\frac{1}{2^n}}^{\frac{q}{2}}}^{\frac{2}{q}-2}}^{\frac{1}{2}}\leq \\
\leq 2^{\frac{q-1}{2}}\pa{\sum_{i=1}^{+\infty}\frac{(x,v_i)^2}{2^i}\pa{(x,v_i)^2+\frac{1}{2^n}}^{q-2}}^{\frac{1}{2}}\leq 2^{\frac{q-1}{2}}\pa{\sum_{i=1}^{+\infty}\frac{(x,v_i)^{2q-2}}{2^i}}^{\frac{1}{2}}\leq
2^{\frac{q-1}{2}}\pa{\sum_{i=1}^{+\infty}(x,v_i)^{q-1}}.
\end{gather*}
By Proposition \ref{norma integrabile} the last function is integrable for every $p\geq 1$. 
If $p>2$ then
\begin{gather*}
\abs{\nabla_H U_n(x)}_H=\pa{\sum_{i=1}^{+\infty}\frac{(x,v_i)^2}{2^i}\pa{(x,v_i)^2+\frac{1}{2^n}}^{q-2}\pa{\sum_{k=1}^n\pa{(x,v_k)^2+\frac{1}{2^n}}^{\frac{q}{2}}}^{\frac{2}{q}-2}}^{\frac{1}{2}}\leq \\
\leq 2^{\frac{q-1}{2}}\pa{\sum_{i=1}^{+\infty}\frac{(x,v_i)^2}{2^i}\pa{(x,v_i)^2+\frac{1}{2^n}}^{q-2}}^{\frac{1}{2}}\leq\\
\leq 2^{\frac{q-1}{2}}\max\set{1,2^{\frac{q-3}{2}}}\pa{\sum_{i=1}^{+\infty}\pa{(x,v_i)^{2q-2}+\frac{(x,v_i)^{4}}{2^{n(q-2)}}}}^{\frac{1}{2}}\leq\\
\leq 2^{\frac{q-1}{2}}\max\set{1,2^{\frac{q-3}{2}}}\pa{\sum_{i=1}^{+\infty}(x,v_i)^{q-1}+\sum_{i=1}^{+\infty}(x,v_i)^{2}}.
\end{gather*}
By Proposition \ref{norma integrabile} the last term belongs to $\elle^p(\ell_2,\mu)$ for every $p\geq 1$. By the Lebesgue dominated convergence theorem we get $U\in W^{1,p}(\ell_2,\mu)$ for every $p\geq 1$ and $\partial_i U(x)=2^{-\frac{i}{2}}\sign(x,v_i)(x,v_i)^{q-1}\norm{x}_q^{1-q}$ $\mu$-a.e. for every $i\in\N$.
\end{proof}

By Proposition \ref{C1 func} (applied with the functions $\theta_n(\eta)=n\arctan(n^{-1}e^\eta)$, for every $\eta\in\R$ and $n\in\N$, and then using the Lebesgue dominated convergence theorem) we get
\begin{gather*}
\partial_i w_q(x)=2^{-\frac{i}{2}}\sign(x,v_i)(x,v_i)^{q-1}\norm{x}_q^{1-q}e^{\norm{x}_q};\\
\partial_i \log w_q(x)=2^{-\frac{i}{2}}\sign(x,v_i)(x,v_i)^{q-1}\norm{x}_q^{1-q}.
\end{gather*}
Proposition \ref{Fernique without square} and Proposition \ref{norma e' sobolev} yield that the function $w_q$ satisfies Hypothesis \ref{ipotesi1} for every $s,t>1$. This implies that it is possible to define $W^{1,p}(\ell_2,\nu_q)$ for every $p>1$ (see Definition \ref{Weightes Sobolev space definition}). Observe that $\norm{\cdot}_q$ is not $\mu$-a.e. continuous on $\ell_2$ if $q\in(1,2)$.

Using the same argument already used in Example \ref{esempio spazio hilbert} it is possible to prove that $S$ satisfies Hypothesis \ref{ipotesi2}. The trace operator $\trace_{S^{-1}(0)}$ maps $W^{1,p}(S^{-1}(-\infty,0),\nu_q)$ into $\elle^r(S^{-1}(0),w_q\rho)$ for every $r\in[1,p)$. We do not know whether its range is contained in $\elle^p(S^{-1}(0),w_q\rho)$. The same considerations are true for the trace operator $\trace_{G_f^{-1}(0)}$ for every $f\in\ell_2\ssm\set{0}$.

\subsection{An example in $\con[0,1]$}\label{esempio C0}

Recall that a function $f:X\ra\R$ from a Banach space $X$ to $\R$ is G\^ateaux differentiable in $x\in X$ if for every $y\in X$ the limit
\[\lim_{t\ra 0}\frac{f(x+ty)-f(x)}{t}\]
exists and defines a linear (in $y$) map $((Df)x)(\cdot)$ which is continuous from $X$ to $\R$. 

We will use the following result of Aronszajn (see \cite[Theorem 1 of Chapter 2]{Aro76} and \cite[Theorem 6]{Phe78}).

\begin{thm}\label{Aronszajn}
Suppose that $X$ is a separable real Banach space. If $f:X\ra\R$ is a continuous convex function, then $f$ is G\^ateaux differentiable outside a Gaussian null set, i.e. a Borel set $A\subseteq X$ such that $\mu(A)=0$ for every nondegenerate Gaussian measure $\mu$ on $X$.
\end{thm}

Consider the classical Wiener measure $P^W$ on $\con[0,1]$ (see \cite[Example 2.3.11 and Remark 2.3.13]{Bog98} for its construction). Recall that the Cameron--Martin space is the space of the continuous functions $f$ on $[0,1]$ such that $f$ is absolutely continuous, $f'\in\elle^2[0,1]$ and $f(0)=0$. In addition $\abs{f}_H=\norm{f'}_{\elle^2[0,1]}$ (see \cite[Lemma 2.3.14]{Bog98}). An orthonormal basis of $H$ is given by the functions
\[f_n(\xi)=\sqrt{2\lambda_n}\sin\frac{\xi}{\sqrt{\lambda_n}}\qquad\text{where }\lambda_n=\frac{4}{\pi^2(2 n-1)^2}\text{ for every }n\in\N.\]

Consider the weight 
\[w(f)=e^{\norm{f}_\infty}\qquad\text{for every }f\in\con[0,1].\] 
According to \cite[Theorem 5.11.2]{Bog98} $w$ is differentiable along $H$. In particular letting 
\[M=\set{f\in\con[0,1]\tc \text{ there exists an unique $\xi\in[0,1]$ such that }\norm{f}_\infty=\abs{f(\xi)}},\]
by Theorem \ref{Aronszajn} and \cite[Example 1.6.b]{DGZ93} we get that $P^W\pa{M}=1$. Furthermore
by \cite{Ban87} 
\begin{gather*}
((D\norm{\cdot}_\infty)f)(g)=\sign(f(\xi_f))g(\xi_f),
\end{gather*}
for every $f\in M$ and $g,g_1,g_2\in\con[0,1]$, where $\xi_f\in[0,1]$ is the only point of maximum of the function $\abs{f(\cdot)}$. By \cite[Definition 5.2.3 and Proposition 5.4.6(iii)]{Bog98} and by Proposition \ref{C1 func} (applied with the functions $\theta_n(\eta)=n\arctan(n^{-1}e^\eta)$, for every $\eta\in\R$ and $n\in\N$, and then using the Lebesgue dominated convergence theorem) it can be seen that
\begin{gather*}
(\partial_i w(f))(\xi)=e^{\norm{f}_\infty}\sign(f(\xi_f))f_i(\xi_f);\\
(\partial_i \log w(f))(\xi)=\sign(f(\xi_f))f_i(\xi_f).
\end{gather*}
We also have
\begin{gather}
\label{primo integrale}
\norm{w}_{\elle^s(\con[0,1],P^W)}=\pa{\int_{\con[0,1]}e^{s\norm{f}_\infty}dP^W(f)}^{\frac{1}{s}};\\
\label{secondo integrale}
\norm{\nabla_H w}_{\elle^s(\con[0,1],P^W; H)}=\pa{\int_{\con[0,1]}e^{s\norm{f}_\infty}\pa{\sum_{n=1}^{+\infty}f_n^2(\xi_f)}^sdP^W(f)}^{\frac{1}{s}};\\
\label{terzo integrale}
\norm{\log w}_{\elle^t(\con[0,1],P^W)}=\pa{\int_{\con[0,1]}\norm{f}^t_\infty dP^W(f)}^{\frac{1}{t}};\\
\label{quarto integrale}
\norm{\nabla_H \log w}_{\elle^t(\con[0,1],P^W;H)}=\pa{\int_{\con[0,1]}\pa{\sum_{n=1}^{+\infty}f^2_n(\xi_f)}^t dP^W(f)}^{\frac{1}{t}}.
\end{gather}
By Proposition \ref{Fernique without square}, (\ref{primo integrale}) and (\ref{terzo integrale}) are finite for every $s,t>1$.
By \cite[Theorem 5.11.2]{Bog98}, (\ref{secondo integrale}) and (\ref{quarto integrale}) are finite for every $s>1$. 
All these results give that the weight $w$ satisfies Hypothesis \ref{ipotesi1} for every $s,t>1$. This implies that it is possible to define $W^{1,p}(\con[0,1],\nu)$ for every $p>1$ (see Definition \ref{Weightes Sobolev space definition}).

If we let 
\[S_1(f)=\norm{f}_2-1,\] 
then $S_1\in\bigcap_{p>1} W^{2,p}(\con[0,1],P^W)$.
Furthermore $\nabla_H S_1(f)=\sum_{i=1}^{+\infty}\frac{\int_0^1f(\xi)f_i(\xi)d\xi}{\norm{f}_2}f_i$ and
\begin{gather}\label{altro integrale}
\int_{\con[0,1]}\frac{1}{\abs{\nabla_H S_1(f)}_H^p}dP^W(f)
=\int_{\con[0,1]}\frac{\norm{f}_2^p}{(\sum_{j=1}^{+\infty}(\int_0^1f(\xi)f_n(\xi)d\xi)^2)^{\frac{p}{2}}}dP^W(f).
\end{gather}
Using the H\"older inequality with some $\alpha>1$ we get
\begin{gather*}
\int_{\con[0,1]}\frac{1}{\abs{\nabla_H S_1(f)}_H^p}dP^W(f)\leq \pa{\int_{\con[0,1]}\norm{f}_2^{p\alpha'}dP^W(f)}^{\frac{1}{\alpha'}}\pa{\int_{\con[0,1]}\frac{dP^W(f)}{(\sum_{j=1}^{+\infty}(\int_0^1f(\xi)f_n(\xi)d\xi)^2)^{\frac{p\alpha}{2}}}}^{\frac{1}{\alpha}}.
\end{gather*}
Finally $\int_{\con[0,1]}\norm{f}_2^{p\alpha'}dP^W(f)$ is finite by Fernique's theorem (see \cite[Theorem 2.8.5]{Bog98}), and recalling that $\elle^2$-continuous linear functional are also $\con[0,1]$-continuous we can use the change of variable formula (see \cite[Equation (A.3.1)]{Bog98}) and obtain
\begin{gather*}
\int_{\con[0,1]}\frac{dP^W(f)}{(\sum_{j=1}^{+\infty}(\int_0^1f(\xi)f_n(\xi)d\xi)^2)^{\frac{p\alpha}{2}}}\leq \int_{\con[0,1]}\frac{dP^W(f)}{(\sum_{j=1}^{N}(\int_0^1f(\xi)f_n(\xi)d\xi)^2)^{\frac{p\alpha}{2}}}=\\
=\frac{1}{\Gamma\pa{\frac{p\alpha}{2}}}\int_0^{+\infty}\xi^{\frac{p\alpha}{2}-1}\prod_{i=1}^N\frac{d\xi}{\sqrt{1+2\xi\lambda_i^2}}.
\end{gather*}
For $N\in\N$ big enough the last integral is finite and we get that (\ref{altro integrale}) is finite for every $p>1$. Thus $S_1$ satisfies Hypothesis \ref{ipotesi2}. 
By Corollary \ref{Trace continuity}, the trace operator $\trace_{S^{-1}(0)}$ maps the space $W^{1,p}(S_1^{-1}(-\infty,0),\nu)$ into $\elle^q(S_1^{-1}(0),w\rho)$ continuously, for every $q\in [1,p)$.

Fix a finite signed Borel measure $\lambda$ on $[0,1]$ and consider the continuous linear functional 
\[G_\lambda(f)=\int_0^1f(x)d\lambda(x)\qquad\text{for every }f\in\con[0,1].\] 
Using a similar argument as in Proposition \ref{proprerties of sobolev space}, we get $\trace_{G_\lambda^{-1}(0)}\varphi\in\elle^q(G_\lambda^{-1}(0),\rho)$ for every $q\in[1,p)$ and every $\varphi\in W^{1,p}(G^{-1}_\lambda(-\infty,0),\nu)$.

\bibliographystyle{plain}
\nocite{*} 
\bibliography{bibpesi}

\end{document}